\documentclass[11pt]{article}
\usepackage[final]{epsfig}
\usepackage{graphics}
\usepackage[left=3cm,right=3cm, top=2.5cm,bottom=2.5cm,bindingoffset=0cm]{geometry}
\usepackage{amsmath}
\usepackage{amsfonts}
\usepackage{amsthm}
\usepackage{latexsym}
\usepackage{amssymb, mathabx}
\usepackage{epstopdf, setspace}
\usepackage{multicol,multirow, enumitem}
\usepackage{hyperref, mathtools}
\newlist{longenum}{enumerate}{5}
\setlist[longenum,1]{label=\roman*)}
\setlist[longenum,2]{label=\alph*)}
 
\usepackage{wrapfig}

\usepackage{calrsfs}
\DeclareMathAlphabet{\pazocal}{OMS}{zplm}{m}{n}

\usepackage{tikz}
\usetikzlibrary{positioning}
\usetikzlibrary{decorations.text}
\usetikzlibrary{decorations.pathmorphing}
\usetikzlibrary{decorations.markings}

\usetikzlibrary{arrows} 
\tikzset{
    >=stealth',
    pil/.style={
           ->,
           thick,
           shorten <=2pt,
           shorten >=2pt,}
}
\tikzset{->-/.style={decoration={
  markings,
  mark=at position .7 with {\arrow{>}}},postaction={decorate}}}
\tikzset{-->-/.style={decoration={
  markings,
  mark=at position .75 with {\arrow{>}}},postaction={decorate}}}
  \tikzset{->--/.style={decoration={
  markings,
  mark=at position .3 with {\arrow{>}}},postaction={decorate}}}
  \tikzset{a/.style={decoration={
  markings,
  mark=at position -1.5pt with {\arrow{>}}},shorten >=2pt, postaction={decorate}}}
\tikzset{-<-/.style={decoration={
  markings,
  mark=at position .4 with {\arrow{<}}},postaction={decorate}}}  

\tikzstyle{vertex} = [coordinate]

\newtheorem{lemma}{Lemma}[section]
\newtheorem{proposition}[lemma]{Proposition}
\newtheorem{remark-definition}[lemma]{Remark-Definition}
\newtheorem{theorem}[lemma]{Theorem}
\newtheorem{corollary}[lemma]{Corollary}

\newtheorem{proposition-conjecture}[lemma]{Proposition-conjecture}
\newtheorem{problem}[lemma]{Problem}

\theoremstyle{definition}
\newtheorem{example}[lemma]{Example}

\newtheorem{definition}[lemma]{Definition}
\newtheorem{remark}[lemma]{Remark}
\begin{document}
\newcommand{\eps}{{\varepsilon}}
\newcommand{\proofend}{\hfill$\Box$\bigskip}
\newcommand{\C}{{\mathbb C}}
\newcommand{\Q}{{\mathbb Q}}
\newcommand{\R}{{\mathbb R}}
\newcommand{\Z}{{\mathbb Z}}
\newcommand{\RP}{{\mathbb {RP}}}
\newcommand{\CP}{{\mathbb {CP}}}
\newcommand{\PP}{{\mathbb {P}}}
\newcommand{\ep}{\epsilon}
\newcommand{\G}{{\Gamma}}

\newcommand{\SDiff}{\mathrm{SDiff}}
\newcommand{\SVect}{\mathfrak{s}\mathfrak{vect}}
\newcommand{\vect}{\mathfrak{vect}}

\newcommand{\Ker}[1]{\mathrm{Ker} \, #1}
\newcommand{\grad}[1]{\mathrm{grad} \, #1}
\newcommand{\sgrad}[1]{\mathrm{sgrad} \, #1}
\newcommand{\St}[1]{\mathrm{St} \, #1}
\newcommand{\rank}[1]{\mathrm{rank} \, #1}
\newcommand{\codim}[1]{\mathrm{codim} \, #1}
\newcommand{\corank}[1]{\mathrm{corank} \, #1}
\newcommand{\sgn}[1]{\mathrm{sign} \, #1}
\newcommand{\ann}[1]{\mathrm{ann} \, #1}
\newcommand{\ind}[1]{\mathrm{ind} \, #1}
\newcommand{\Ann}[1]{\mathrm{Ann} \, #1}
\newcommand{\ls}[1]{\mathrm{span} \langle #1 \rangle}
\newcommand{\Tor}[1]{\mathrm{Tor} \, #1}
\newcommand{\diff}[1]{\mathrm{d}  #1}
\newcommand{\diffFX}[2]{ \dfrac{\partial #1}{\partial #2} }
\newcommand{\diffFXp}[2]{ \dfrac{\diff #1}{\diff #2} }
\newcommand{\diffX}[1]{ \frac{\partial }{\partial #1} }
\newcommand{\diffXp}[1]{ \frac{\diff }{\diff #1} }
\newcommand{\diffFXY}[3]{ \frac{\partial^2 #1}{\partial #2 \partial #3} }
\newcommand{\K}{\mathbb{K}}
\newcommand{\centrum}{\mathrm{Z}}
\newcommand{\Complex}{\mathbb{C}}
\newcommand{\Aut}{\mathrm{Aut}}
\newcommand{\Id}{\mathrm{E}}
\newcommand{\D}{\mathrm{D}}
\newcommand{\T}{\mathrm{T}}
\newcommand{\Cont}{\mathrm{C}}
\newcommand{\const}{\mathrm{const}}
\newcommand{\Hom}{\mathrm{H}}
\newcommand{\Ree}[1]{\mathrm{Re} \, #1}
\newcommand{\Imm}[1]{\mathrm{Im} \, #1}
\newcommand{\Tr}[1]{\mathrm{Tr} \, #1}
\newcommand{\tr}[1]{\mathrm{tr} \, #1}
\newcommand{\matrixtwobytwo}{\left(\begin{array}{|cc|}\hline 0 & 0 \\0 & 0 \\\hline \end{array}\right)}
\newcommand{\wave}{\tilde}
\newcommand{\LieBracket}{ [\, , ] }
\newcommand{\PoissonBracket}{ \{ \, , \} }
\newcommand{\g}{\mathfrak{g}}
\newcommand{\h}{\mathfrak{h}}
\newcommand{\lCal}{\mathfrak{l}}
\newcommand{\e}{\mathfrak{e}}
\newcommand{\so}{\mathfrak{so}}
\newcommand{\SO}{\mathrm{SO}}
\newcommand{\Orth}{\mathrm{O}}
\newcommand{\U}{\mathrm{U}}
\newcommand{\he}{\mathfrak{hso}}
\newcommand{\ELL}{\mathfrak{D}}
\newcommand{\hyp}{\mathfrak{D}^{h}}
\newcommand{\foc}{\mathfrak{D}^{\Complex}}
\newcommand{\sP}{\mathfrak{sp}}
\newcommand{\sL}{\mathfrak{sl}}
\newcommand{\ad}{\mathrm{ad}}
\newcommand{\Ad}{\mathrm{Ad}}
\newcommand{\zenter}{\mathrm{Z}}
\newcommand{\id}{\mathrm{id}}
\newcommand{\Ham}{\mathrm{Ham}}
\newcommand{\ham}{\mathfrak{ham}}
\newcommand{\Flux}{\mathrm{Flux}}
\newcommand{\Diffeo}{\mathrm{Diff}}
\newcommand{\halfTwist}{\mathrm{ht}}
\newcommand{\closure}{\wave}

\renewcommand{\proofname}{Proof}

%Notation for 1-forms
\newcommand{\oneform}{\alpha}
\newcommand{\oneformtwo}{\beta}
\newcommand{\oneformthree}{\gamma}

% Notation for circulation function
\newcommand{\circulation}{ \mathfrak c}
\newcommand{\circulationone}{ \circulation}
\newcommand{\circulationtwo}{ \wave \circulation}
% Notation for coadjoint orbit
\newcommand{\orbit}{\pazocal O}

%Notation for fibration
\newcommand{\Fibr}{\pazocal{F}}
\newcommand{\Fibrtwo}{\pazocal{G}}

\newcommand{\Diff}{\mathfrak{curl}}

\newcommand{\MCG}{\mathrm{Mod}}
\newcommand{\Stab}{\mathrm{Stab}}

\sloppy

\newcounter{bk}
\newcommand{\bk}[1]
{\stepcounter{bk}$^{\bf BK\thebk}$%
\footnotetext{\hspace{-3.7mm}$^{\blacksquare\!\blacksquare}$
{\bf BK\thebk:~}#1}}

\newcounter{ai}
\newcommand{\ai}[1]
{\stepcounter{ai}$^{\bf AI\theai}$%
\footnotetext{\hspace{-3.7mm}$^{\blacksquare\!\blacksquare}$
{\bf AI\theai:~}#1}}

\newcommand*\circled[1]{\tikz[baseline=(char.base)]{
            \node[shape=circle,draw,inner sep=0.8pt] (char) {#1};}}

            \mathtoolsset{showonlyrefs}

%%%%%%%%%%%%%%%%%%%%%%%%%%%%%%%%%%%%%%%%%%%%%%%%%%%%%%%%%%

\title{Characterization of steady solutions  to the 2D Euler equation}

\author{Anton Izosimov\thanks{
Department of Mathematics,
University of Toronto, Toronto, ON M5S 2E4, Canada;
e-mail: {\tt izosimov@math.toronto.edu} and \tt{khesin@math.toronto.edu}
} \,
and Boris Khesin$^*$}
%\thanks{ Department of Mathematics,
%University of Toronto, Toronto, ON M5S 2E4, Canada;
%e-mail: \tt{khesin@math.toronto.edu}}
%\\
%}
%\affil{Department of Mathematics, University of Toronto, Toronto, ON M5S 2E4, Canada}

\date{\vspace{-5ex}}
%\linespread{1.1}
\maketitle

\begin{abstract} 
Steady fluid flows have very special topology. In this paper we describe necessary and sufficient conditions on the vorticity 
function of a 2D ideal flow on a surface with or without boundary, for which there exists a steady flow among isovorticed fields.
 For this we introduce the notion of an antiderivative (or  circulation function) 
on a measured graph, the Reeb graph associated 
to the vorticity function on the surface, while the criterion is related to the total negativity of this antiderivative.
It turns out that given topology of the vorticity function, the set of coadjoint orbits of the symplectomorphism group 
admitting steady flows  with this topology forms a convex polytope. 
As a byproduct of the proposed construction,  we also describe a complete list of Casimirs for the 2D Euler hydrodynamics: 
we define  generalized enstrophies which, along with  circulations,  form a complete set of invariants for coadjoint orbits of area-preserving diffeomorphisms on a surface.
\end{abstract}

\tableofcontents

\bigskip

%%%%%%%%%%%%%%%%%%%%%%%%%%%%%%%%%%

\section{Introduction} \label{sect:intro}

The study of topology of steady fluid flows is one of central topics of ideal hydrodynamics. 
Steady flows are key objects in a variety of questions related to hydrodynamical stability, existence of attractors, in
dynamo constructions and magnetohydrodynamics, etc. Such flows satisfy very special constraints 
and are difficult to construct explicitly (beyond certain very symmetric settings). 

In this paper we present  necessary conditions for the existence of steady flows 
with a given vorticity function on an arbitrary 2D surface, possibly with boundary. 
Furthermore, we show that these conditions are also sufficient in the sense that for any vorticity satisfying these conditions
there is a steady flow for an appropriate choice of metric. This solves one of Arnold's problems raised in
\cite{ArnPr} about existence of smooth minimizers
in Dirichlet problems for initial functions of various topology.

One is usually looking for Euler stationary solutions among fields with prescribed vorticity
since vorticity of the fluid is frozen into the flow. Our main tool in the study of isovorticed fields is 
a description of fine properties of the measured Reeb graphs of vorticity functions
and the notion of a graph antiderivative. In this hydrodynamical application we call this antiderivative a circulation function.
As we prove below it turns out that the criterion for the existence of a steady flow is that the corresponding circulation function
is balanced, i.e. circulations in the vicinity of every singular level have the same signs, as we explain below. Furthermore,
given the geometry of the vorticity function, it turns out that the set of coadjoint orbits admitting  a stationary 
Euler flow forms a convex polytope. This opens a variety of questions related to deeper connections between steady 
flows and convex geometry.

\begin{example}\label{ex:intro}
 The following example illustrates the results discussed below in a nutshell.
% \ai{Change to ``illustrates the notion of a measured Reeb graph''? In fact, I am not sure we need this example here. May be we should concentrate on and explain in more detail the next example (Figure \ref{pretzel}). I will adapt Figure \ref{torusInt} if we decide to leave it.}
Consider a generic vorticity function on a surface $M$, see Figure \ref{pretzel} where vorticity is the height function
on a pretzel. (For a surface with boundary we consider a vorticity function constant on the boundary $\partial M$.)

Associate with this function its Reeb graph $\Gamma_F$,  which is the graph representing the
set of connected components of  levels of the height vorticity function $F$. 
Critical points of $F$ correspond to the vertices 
 of the graph $\Gamma_F$. This graph comes with a natural parametrization by the values of $F$. 
Furthermore, the surface $M$ is equipped with an area form, preserved by the fluid motion.
This area form on the surface gives rise to a log-smooth measure on the Reeb graph $\Gamma_F$.
It turns out that such a measured Reeb graph is the only invariant of a simple Morse function on a surface modulo
symplectomorphisms. 
%if two such graphs are isomorphic, the corresponding functions can be related by a composition with an area-preserving diffeomorphism of the surface preserving boundary.
\end{example}
\begin{figure}[t]
\centerline{
\begin{tikzpicture}[thick, scale = 1.5]
% \fill (2.8, 3.7) circle [radius=1.5pt];
  \node at (2.5,2) ()
  {
  \begin{tikzpicture}[thick, scale = 1]
    \draw (2.4,1.9) ellipse (1.4cm and 2.6cm);
   \node at (2.7,1.9) () { 
     \begin{tikzpicture}[thick, yscale = 2]
    \draw   (2.85,1.13) arc (260:100:0.8cm);
    \draw   (2.6,1.23) arc (-80:80:0.7cm);
    \end{tikzpicture}
    };
       \node at (1.7,1.9) () { 
     \begin{tikzpicture}[thick, yscale = 1.5]
    \draw   (1.75,1.48) arc (260:100:0.4cm);
    \draw   (1.55,1.58) arc (-80:80:0.3cm);
    \end{tikzpicture}
    };
        \draw  [->] (0.7,0.6) -- (0.7,3.4);
    \node at (0.4,2) (nodeA) {$F$};
    \end{tikzpicture}
    };
%    \draw   (1.1,1.05) arc (-120:-60:0.95cm);
%    \draw   (2.05,1.05) arc (-120:-60:0.85cm);
%    \draw  [densely dashed] (2.05,1.05) arc (60:120:0.95cm);
%    \draw  [densely dashed] (2.9,1.05) arc (60:120:0.85cm);
    %
%    \draw   (1.1,2.75) arc (-120:-60:0.95cm);
%    \draw  (2.05,2.75) arc (-120:-60:0.85cm);
%    \draw  [densely dashed] (2.05,2.75) arc (60:120:0.95cm);
%    \draw  [densely dashed] (2.9,2.75) arc (60:120:0.85cm);
    \node [vertex] at (5.5,0.3) (nodeC) {};
    \node [vertex]  at (5.5,1) (nodeD) {};
    \node at (5.3,0.7) () {{$-1$}};
        \node at (5.1,1.3) () {$-1$};
                \node at (6.05,2) () {$0$};
                                \node at (5.35,2) () {$0$};
                                  \node at (4.85,2) () {$0$};
                                  \draw [->, dashed]  (4,2) -- (4.5,2);
                                   \draw [->, dashed]  (6.8,2) -- (7.3,2);
                                          \node at (5.2,2.7) () {$1$};
                                              \node at (5.4,3.4) () {$1$};
    \node [vertex] at (5.5,3) (nodeE) {};
    \node [vertex]  at (5.5,3.7) (nodeF) {};
    \node [vertex]  at (5.1,1.7) (nodeG) {};    
        \node []  at (5,1.6) () {$v$};    
        \node [vertex]  at (5.1,2.3) (nodeH) {};    
    \draw  [a] (nodeC) -- (nodeD);
    \fill (nodeC) circle [radius=1.5pt];
    \fill (nodeD) circle [radius=1.5pt];
    \fill (nodeE) circle [radius=1.5pt];
    \fill (nodeF) circle [radius=1.5pt];
        \fill (nodeG) circle [radius=1.5pt];
            \fill (nodeH) circle [radius=1.5pt];
    \draw  [a] (nodeE) -- (nodeF);
    \draw[a] (nodeG)  .. controls +(-0.2,+0.3) ..  (nodeH);
        \draw[a] (nodeG)  .. controls +(0.2,+0.3) ..  (nodeH);
                \draw[a] (nodeD) -- (nodeG);
        \draw[a] (nodeH) -- (nodeE);
     \draw[a] (nodeD)  .. controls +(0.6,+1) ..  (nodeE);     
             \node at (4.7, 3) () {$\Gamma_F$};  
                    \node at (2, 3.6) () {$M$};  
\node at (8.5,2.2) () {
\begin{tikzpicture}[thick, scale = 1.3]
\draw (9,3) -- (10,3) -- (11,2) -- (11,1) -- (9,3);
 \fill[opacity = 0.3] (9,3) -- (10,3) -- (11,2) -- (11,1) -- (9,3);
 \node at (8.9,3.15) () {$-2$};
  \node at (9.9,3.15) () {$-1$};
    \node at (11.3,2) () {$-1$};
        \node at (11.3,1) () {$-2$};
 \draw [->] (8.5,3) -- (12,3);
 \node at (11.8, 2.8) () {$c_1(v)$};
  \node at (11.4, 3.9) () {$c_2(v)$};
    \node at (10, 3.9) () {$\Hom_1(\Gamma_F, \R)$};
 \draw [->] (11,0.5) -- (11,4);
     \draw [decoration={text along path,
    text={steady flows},text align={center}, raise = -0.1cm},decorate]    (9.5,3) -- (11,1.5); 
 \end{tikzpicture}
 };
\end{tikzpicture}
}
\caption{A vorticity $F$ on a pretzel, its measured Reeb graph, and the corresponding polytope of totally negative circulation functions.}\label{pretzel}
\end{figure}
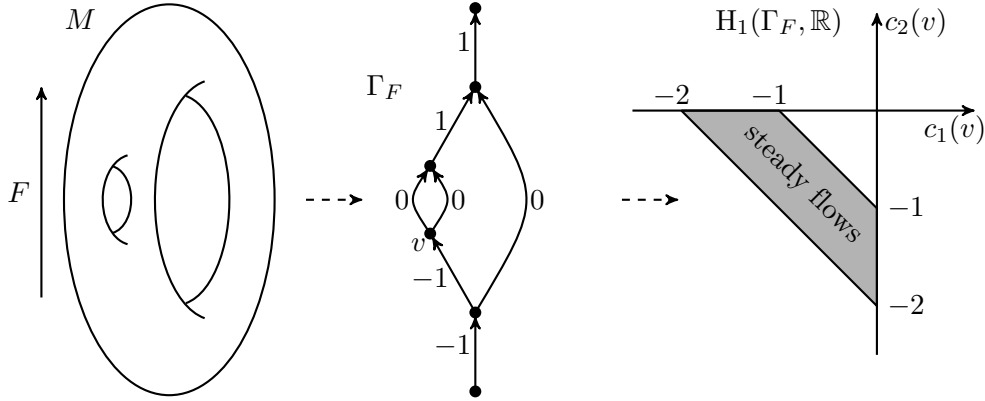

\medskip

{\bf Theorem A (= Theorem \ref{thm:sdiff-functions}).}
{\it The mapping  assigning the measured Reeb graph $\Gamma_F$ to a simple Morse function $F$ provides a one-to-one correspondence between simple Morse functions on $M$ up to  symplectomorphisms 
and measured Reeb graphs compatible with $M$.}
\medskip

Steady flows are known to be critical points of the energy functional restricted to the sets of isovorticed fields, 
i.e. to coadjoint orbits 
of the group of volume- or area-preserving diffeomorphisms of the manifold. 
To describe coadjoint orbits  of the group of symplectomorphisms of a surface $M$ 
we introduce the notion of an antiderivative 
on a graph. Unlike integration on a segment, for a  measured oriented graph the antiderivatives form a space
of dimension equal to the first Betti number of the graph itself, see Section \ref{sect:antider}. (In the case with boundary, the space of antiderivatives has 
dimension of the corresponding relative homology group.) For a pretzel on Figure \ref{pretzel} this space
is two-dimensional, parametrized by circulations over two independent cycles. The space of such antiderivatives, called 
circulation functions, on a graph is in one-to-one correspondence with the space of coadjoint orbits for a given vorticity. 
Roughly speaking, an antiderivative of a measure is a function which is a classical antiderivative at any smooth point 
of the graph,  vanishes at the points of maximum and minimum, while at vertices it can be discontinuous but  satisfies the Kirchhoff  rule. In particular, at any trivalent point in the Reeb graph
the total value of ``incoming" circulations must be equal to the total values of ``outgoing" ones.
One of the main results of the paper fully characterizes which circulation functions correspond to steady flows.
\medskip

{\bf Theorem B (= Theorem \ref{thm:steady}).}
{\it Let $(M, \omega)$ be a compact  symplectic surface, possibly with boundary.
Then Morse-type coadjoint orbit corresponding to a given vorticity function  admits a steady flow if and only if the corresponding circulation function has the same signs near every singular level of the vorticity.}
\medskip

For closed surfaces one can sharpen this description:
\medskip

{\bf Theorem C (= Theorems \ref{thm:steadyclosed} and \ref{thm:polytope}).}
{\it For a symplectic surface without boundary a Morse-type coadjoint orbit for a given vorticity function $F$ admits a steady flow if and only if the corresponding circulation function is totally negative.
Moreover, the set of such totally negative circulation functions is an open convex polytope in the affine space of all circulation functions. In particular, the set of totally negative circulation functions on the Reeb graph $\Gamma_F$ is bounded.
}
\medskip

 The latter implies that ``most orbits"  do not admit any smooth steady flow. 
For instance, consider again the vorticity represented by the height function $F$ of a pretzel, see Figure \ref{pretzel}.
For each edge $e$ of the Reeb graph $\Gamma_F$ consider the domain $M_e\subset M$, the preimage of 
this edge bounded by the corresponding critical levels of $F$. Consider the mean values $\int_{M_e} \!\! F\,\omega$
of the vorticity function in these regions and mark them on the edges of the graph. Assume these mean values of $F$
are as marked on the graph edges in Figure \ref{pretzel}. In order to describe the set of circulation functions
admitting a steady flow we note that the set of all circulation functions on $\Gamma_F$ for given $F$
is parametrized, e.g.,  by limits of circulations $c_1(v), c_2(v)$ at branches of the vertex $v$. (The numbers $c_i(v)$ 
can be thought of as circulations of the fluid velocity along two cycles across two left handles of the pretzel.) 
It turns out that for such a vorticity  the set of orbits admitting 
Euler steady solutions is given only by the pairs of circulations lying inside the shaded trapezoid, see Section \ref{sect:closed}.

\medskip

To summarize, it turns out that there are very restrictive conditions on the sets of isovorticed fields, so that among them there existed a steady flow. 
%We say that  a coadjoint orbit admits a stationary solution if there is a metric compatible with the symplectic structure, for which the Euler equation has a stationary point on this orbit. Then Theorem  \ref{thm:steady} gives a necessary and sufficient condition for an orbit to admit a steady solution on any two-dimensional surface (with or without boundary): its circulation function must be balanced. For surfaces without boundary this can be further specialized: in the latter case according to Theorem~\ref{thm:steadyclosed}  the necessary and sufficient condition is the total negativity of the  circulation function. 
It is interesting to compare these results with topological (rather than symplectic) necessary conditions 
on vorticity functions to admit steady flows obtained in \cite{GK}. 
In the 2D  case they follow from results of this paper, see Example \ref{GKEx}. In a sense, in \cite{GK} the restrictions are obtained 
by using only the information on the Morse indices of the vorticity, i.e. on its Reeb graph, while in this paper 
we show that the use of a measured Reeb graph allows one to sharpen the necessary conditions and, in addition, to obtain
sufficient conditions on the existence of steady flows.
On the other hand, in the case of a simple laminar flow in an annulus, any vorticity without critical points and constant 
on the boundary produces a steady flow on the corresponding orbit (Corollary \ref{cor:no3val}), 
in a perfect matching with results of \cite{ChSv}.
It would be also interesting to find the relation of steady flows with infinite-dimensional convex geometry related to the group
of symplectomorphisms of an annulus \cite{BFR} and of arbitrary toric manifolds \cite{BHFR}.

\medskip

It is worth mentioning one more byproduct of the use of measured Reeb graphs in the description of coadjoint orbits 
of symplectomorphism groups: it allows one to give a complete list of Casimirs in two-dimensional hydrodynamics.
It is known that the 2D Euler equation admits an infinite number of conserved quantities called enstrophies, which are 
moments of the vorticity function. They are called Casimirs, or first integrals of the motion for any  Riemannian 
metric fixed on the surface. However,  the set of all enstrophies is known to be incomplete for flows 
with generic vorticities: there are non-diffeomorphic vorticities with the same values of enstrophies. In  Appendix B 
we show how measured Reeb graphs can be used to define  generalized enstrophies, and prove that 
they together with the set of circulations form a complete list of Casimir invariants for  flows of an ideal 2D fluid
with simple Morse vorticity functions (Corollary \ref{cor:Casimirs}).

To the best of our knowledge, a complete description of Casimirs in 2D fluid dynamics has not previously appeared 
in the literature, while various partial results could be found in \cite{AK, Serre, Shn, Yoshida}.
Note that steady flows are conditional extrema of the energy functional on the sets of isovorticed fields, 
so Casimirs allow one to single out such sets in order to introduce appropriate Lagrange multipliers. 
Furthermore, Casimirs in fluid dynamics are a cornerstone of the energy-Casimir method 
for the study of hydrodynamical stability, see e.g. \cite{Arn66}.

%%%%%%%%%%%%%%%%%%%%%%%%%%%%%%%%%%

\section{Euler equations  and steady flows} \label{sect:euler}

\subsection{Geodesic and Hamiltonian frameworks of the Euler equation}

%In order to describe the framework for steady flows and first integrals we start with the general setting  of an ideal fluid dynamics. 
Consider an inviscid incompressible  fluid filling an $n$-dimensional Riemannian manifold  $M$ 
with the Riemannian volume form $\mu$ and, possibly, with boundary $\partial M$. 
The motion of an ideal fluid is governed by the hydrodynamical Euler equations
\begin{equation}\label{idealEuler}
\begin{aligned}
&\partial_t u+(u, \nabla) u=-\nabla p\,,\\
&{\rm div}\, u=0\quad\text{ and }\quad u\parallel\partial M\,,
\end{aligned}
\end{equation}
describing an evolution of the velocity field $u$ of a fluid flow in $M$, supplemented by the divergence-free condition
and tangency to the boundary. The pressure function $p$ entering the Euler equation 
is defined uniquely modulo an additive constant by the above constraints on the velocity $u$. The notation
$(u, \nabla) u$ stands for the Riemannian covariant derivative $\nabla_u u$ of the field $u$ along itself. 

Arnold  \cite{Arn66} showed that the Euler equations in any dimension can be regarded as an 
equation of the geodesic flow on the group $\SDiff(M):=\{\phi\in \Diffeo(M)~|~\phi^*\mu=\mu\}$ 
of volume-preserving diffeomorphisms of $M$ with respect to the right-invariant $L^2$-metric on the group 
given by the energy of the fluid's velocity field: $E(u)=\frac 12\int_M(u,u)\,\mu$. 
In this approach an evolution of the  velocity field $u(t)$ according to the Euler equations is understood as 
an evolution of the vector in the Lie algebra ${\SVect}(M)=\{u\in {\vect}(M) \mid L_u\mu=0\}$, 
tracing the geodesic on the group $\SDiff(M)$ defined by the given initial condition $u(0)=u_0$ (here $L_u$ stands for the Lie derivative along the field $u$).

%\medskip
\par

The geodesic description implies  the following Hamiltonian framework for the Euler equation.  
Consider the (smooth) dual space $\mathfrak g^*={\SVect}^*(M)$ 
to the space $\mathfrak g={\SVect}(M)$ of divergence-free vector fields on $M$.
This dual space  has a natural  description 
as the  space of cosets  $\mathfrak g^*=\Omega^1(M) / \diff \Omega^0(M)$:
for a 1-form $\alpha$ on $M$ its coset of 1-forms is 
$$
[\alpha]=\{\alpha+df\,|\,\text{ for all } f\in C^\infty(M)\}\in \Omega^1(M) / \diff \Omega^0(M)\,.
$$
The pairing between cosets and divergence-free vector fields is straightforward:
$\langle [\alpha],W\rangle:=\int_M\alpha(W)\,\omega$ for any  field $W\in {\SVect}(M)$.
The coadjoint action of the group $\SDiff(M)$ on the dual 
 $\mathfrak g^*$ is given by the change of coordinates in (cosets of) 1-forms on $M$ 
 by means of volume-preserving diffeomorphisms.
 
The Riemannian metric $(.,.)$ on the manifold $M$ allows one to
identify  (the smooth part of) the Lie algebra and its dual by means of the so-called inertia operator:
given a vector field $u$ on $M$  one defines the 1-form $\alpha=u^\flat$ 
as the pointwise inner product with  the velocity field $u$:
$u^\flat(W): = (u,W)$ for all $W\in T_xM$, see details in \cite{AK}. Note also  that divergence-free fields $u$ correspond to co-closed 1-forms $u^\flat$.
The Euler equation \eqref{idealEuler} rewritten on 1-forms $\alpha=u^\flat$ is
$$
\partial_t \alpha+L_u \alpha=-dP\,
$$
for  an appropriate function $P$ on $M$.
In terms of the cosets of 1-forms $[\alpha]$, the Euler equation on the dual space $\mathfrak g^*$ looks as follows: 
\begin{equation}\label{1-forms}
\partial_t [\alpha]+L_u [\alpha]=0\,.
\end{equation}

The  Euler equation \eqref{1-forms} on $\mathfrak g^*=\SVect^*(M)$ turns out to be a Hamiltonian equation with the
Hamiltonian functional  $\mathcal H$ given by the fluid's kinetic energy,
$$
{\mathcal H}([\alpha])=E(u)= \frac 12\int_M(u,u)\,\mu
$$ 
for $\alpha=u^\flat$.
The corresponding Poisson structure is given by  the 
natural linear Lie-Poisson bracket on the dual space $\mathfrak g^*$ 
of the Lie algebra $\mathfrak g$, see details in \cite{Arn66, AK}.
The corresponding Hamiltonian operator is given by the Lie algebra coadjoint action ${\rm ad}^*_u$, 
which  in the case of the diffeomorphism group corresponds to the Lie derivative: ${\rm ad}^*_u=L_u$.
Its symplectic leaves are coadjoint orbits of the corresponding group $\SDiff(M)$.
All invariants of the coadjoint action, also called Casimirs,  are first integrals of the Euler equation
for {\it any choice} of Riemannian metric.
% One of the goals of this paper is to give a complete description of Casimirs for the 2D Euler equation on surfaces with or without boundary, see Sect. ??.
Stationary (or steady) solutions of the Euler equation satisfying $\partial_t u =0$ (or, equivalently, $\partial_t [\alpha] = 0$) are critical points of the restriction of the 
Hamiltonian functional   $\mathcal H$  to coadjoint orbits.
The main result of this paper is a complete characterization of those coadjoint obits that admit steady solutions.

%Recall the construction of classical enstrophy invariants of the 2D Euler equation.
Recall that  according to the Euler equation  \eqref{1-forms} in any dimension  the 
coset of 1-forms $[\alpha]$ evolves by a volume-preserving change of coordinates, 
i.e. during the Euler evolution it remains in the same coadjoint orbit in $\mathfrak g^*$.
Introduce the {\it vorticity 2-form} $\xi:=\diff u^\flat$ as the differential of the 1-form 
$\alpha=u^\flat$ and note  that  the vorticity exact 2-form is well-defined for cosets $[\alpha]$: 
1-forms $\alpha$ in the same coset have equal vorticities $\xi=\diff \alpha$. 
The corresponding Euler equation assumes the vorticity (or Helmholtz) form
\begin{equation}\label{idealvorticity}
\partial_t \xi+L_u \xi=0\,,
\end{equation}
which means that the vorticity form is transported by (or ``frozen into") the fluid flow (Kelvin's theorem).
The definition of vorticity $\xi$ as an exact 2-form $\xi=\diff u^\flat$ makes sense for a manifold $M$ 
of any dimension. 
%In 3D the vorticity 2-form  is identified with the vorticity vector field $\hat\xi=\mathrm{curl}~v$  by means of the relation $i_{\hat\xi} \mu=\xi$ for the volume form $\mu$.) 

\begin{remark}
In the case of two-dimensional oriented surfaces $M$   the group $\SDiff(M)$ of 
volume-preserving diffeomorphisms of $M$ coincides with the group ${\rm Symp}(M)$ of symplectomorphisms of $M$ 
with the area form $\mu=\omega$ given by the symplectic structure. Furthermore, in 2D we identify the vorticity 2-form
$\xi$ with a function $F$ satisfying $\xi=F\omega$. Note that since locally any symplectic vector field $u$ is Hamiltonian, it 
satisfies $i_u\omega=-\diff H$ for a (locally defined) Hamiltonian function $H$. Then vorticity function $F$ and the 
Hamiltonian $H$ are related via $\Delta H=-F$, by means of the Laplace-Beltrami operator for the Riemannian metric 
on the surface $M$. 
\end{remark}

\medskip

%%%%%%%%%%%%%%%%%%%%%

\subsection{Steady flows and Dirichlet variational problem}

As we mentioned above,  steady flows on a manifold $M$  (of any dimension)  are defined by the Euler stationary equation $\partial_t [\alpha] = 0$, which means
$$
L_u [\alpha]=0\,,
$$ 
where $[\alpha]$ is the coset of the 1-form $\alpha=u^\flat$ obtained from a divergence-free
vector field $u$ on $M$ tangent to the boundary. 
(The Euler stationary equation  is often written in the form $(u,\nabla)u=-\nabla p$ 
for a field $u$ with $\text{div } u=0$, $u\parallel\partial M$, and an appropriate pressure function $p$.)

\begin{remark}
For a two-dimensional surface $M$ steady flows can be described by means of the following  data.
Recall that in the 2D case the vorticity 2-form can be thought of as a function $F=\diff\alpha/\omega$,
and the corresponding  stationary Euler equation assumes the vorticity form $L_uF=0$.
Since  an area-preserving field on a surface $M$ is locally Hamiltonian, 
one can rewrite the stationary condition in terms of the corresponding  
stream (or Hamiltonian) function $H$ of the vector field $u=X_H$.
Namely, the 2D stationary Euler equation $L_uF=0$ assumes the form $\{F, H \}=0$ of vanishing the Poisson bracket 
between  the vorticity $F$ and (local) stream function $H$. The latter in 2D 
means that $F$ and $H$ have  common level curves. 
In particular,  if the Hamiltonian $H=H(F)$ is a smooth function of vorticity $F$ (or vice versa), the corresponding 
Hamiltonian field defines a steady flow.\par
Also note that for steady flows the vorticity $F$ is necessarily constant on connected components of the boundary of $M$, since so is the stream function $H$ (the latter must be constant on boundary components in view of the condition $u\parallel \partial M$).
\end{remark}

In the above description of steady flows one can separate the role of the metric and the area form.
First note that on a surface with a symplectic structure $\omega$ and a Riemannian metric $g(\,,\,)$
there is a natural almost complex structure $J$ relating them: $\omega(V,JW) = g(V,W)$ for any pair of vector fields
$V$ and $W$ on $M$. (On surfaces an 
almost complex structure is always a complex structure for dimensional reasons.) Given only a 
symplectic structure $\omega$ we say that an almost complex structure $J$ is compatible with $\omega$
if the formula above produces a positive-definite metric $g$.

\begin{definition}\label{defn:steady-triple}
Let $(M, \omega)$ be a symplectic surface, and let $F \colon M \to \R$ be a smooth function. Then a triple $(\alpha, J, H)$, where $\alpha$ is a $1$-form, $J$ is an almost complex structure compatible with $\omega$, and $H = H(F)$ is a smooth function of $F$, such that
$J^*\alpha = -\diff  H$ and $\diff \alpha = F  \omega$, is called an \textit{$F$-steady triple}.
\end{definition}

\begin{proposition}
A  steady triple $(\alpha, J, H)$ on a symplectic surface  $(M, \omega)$ 
gives rise to a steady flow with vorticity $F$ and stream function $H$.
\end{proposition}

\begin{proof}
Indeed,  let us show that such a 1-form $\alpha$ defines a steady flow $u=\alpha^\flat$, i.e. that the condition 
$L_u [\alpha]=0$ is fulfilled.  Since $L_u=i_u\diff + \diff i_u$ we need to show that the 1-form $i_u\diff\alpha$ is exact. 
The latter holds, since  $i_u\diff\alpha=i_uF\omega=Fi_u\omega=-F\diff H(F)=\diff\Phi(F)$ for an appropriate function $\Phi$.
Furthermore, the relation 
$J^*\alpha = -\diff  H$ along with the condition of compatibility of $J$ and $\omega$ for the Hamiltonian field, which satisfies 
$i_u\omega=-\diff H$,  gives $\alpha(\,.\,)=g(u, \,.\,)$, i.e. precisely $u=\alpha^\flat$ in the metric $g$.
Thus the Hamiltonian field $u$ for the Hamiltonian function $H$ is a steady solution of the Euler equation.
\end{proof}

It is worth mentioning that, in addition to the direct definition of steady flows
by means of the stationary Euler equation, there are two variational problems leading to the same characterization, both of which 
have hydrodynamical origin. The first description is a direct consequence of the aforementioned Hamiltonian framework:

\begin{proposition}
Steady flows of an ideal fluid  correspond to critical points of the $L^2$-energy Hamiltonian $H([\alpha])$ on the sets of 
isovorticed fields.
\end{proposition}

Indeed, being Hamiltonian on the dual space $\mathfrak g^*$, steady flows are critical points 
of the Hamiltonian function, which is the $L^2$-energy,
on {\it coadjoint orbits} of volume-preserving diffeomorphisms. The latter consist of (cosets of) 1-forms $\alpha$ which
are diffeomorphic by means of volume-preserving diffeomorphisms, and the corresponding vector fields are called 
 isovorticed, since they correspond to diffeomorphic vorticity 2-forms $\diff \alpha$. 

It turns out that steady fields in {\it ideal hydrodynamics} have yet another description as solutions in a variational problem known in {\it magnetohydrodynamics}. 

\begin{proposition} \label{adj-orbit-problem} 
Steady flows of an ideal fluid  correspond to critical points of the $L^2$-energy  $E(u)$ on the sets of 
 fields related by the action of volume-preserving diffeomorphisms.
\end{proposition}

The latter conditional extremum problem is, in a sense, dual to the one above: now we are looking for critical points
of the same energy functional but on the sets of diffeomorphic fields, which are 
{\it adjoint orbits} in the Lie algebra $\mathfrak g$ of the group of volume-preserving diffeomorphisms. 
Nevertheless, the sets of their conditional extrema coincide, see \cite{AK} for the proof and details in a more general context.

In terms of the stream function $H$, the second variational problem 
is closely related to the following minimization problem, posed by V.~Arnold in \cite{Arn74}:

\begin{problem} Given a simple Morse function $\bar H$ on a surface $M$, find the functions from the orbit 
$\mathcal O_{\bar H}= \{H\in C^\infty(M)~|~H=\phi^*{\bar H} \text{ for } \phi\in \SDiff(M)\}$ realizing the minimum 
 for the Dirichlet functional: 
$DF(H):= \int_M (\nabla H, \nabla H)\,\omega$. 
\end{problem}

Comparing this problem with Proposition~\ref{adj-orbit-problem} (and noticing that $DF(H) = \int_M (J\nabla H, J\nabla H)\,\omega = E(u)$)
one can see that functions realizing extremals of the Dirichlet functional 
can be regarded as stream functions of steady flows on $M$.

\begin{example} For a function $\bar H$ on a disk with the only critical point, maximum, inside the disk 
and constant on the boundary, 
the minimum of the Dirichlet functional on the orbit $\mathcal O_{\bar H}$ is assumed on a rotationally symmetric function
\cite{Arn74}. It corresponds to a steady flow on the disk whose stream function depends only on the radius and flow trajectories are concentric circles. On the other hand, for a positive function on the disk with one local maximum, one local minimum, and one saddle point there is no smooth function realizing the
minimum \cite{GK}.
\end{example}

In this paper we show how very restrictive this condition of stationarity is and give a criterion for a
coadjoint orbit to admit a steady solution for an appropriate metric.

\medskip

\subsection{Casimirs of the 2D Euler equation}
The fact that the vorticity 2-form $\xi$ is ``frozen into" the incompressible flow allows one to define first integrals of the  hydrodynamical Euler equation valid for any Riemannian metric on $M$. 
%,  e.g., the conservation of helicity in 3D and  enstrophies in 2D. 
In 2D the Euler equation on $M$ is known to possess infinitely many so-called {\it enstrophy invariants}
$m_\lambda(F):=\int_M \!\lambda(F)\,\omega\,,$
where 
$\lambda(F)$ is an arbitrary function of the vorticity function $F$. In particular,  the enstrophy momenta
$$
m_i(F):=\int_M F^i\,\omega
$$ 
are conserved quantities for any $i\in \mathbb N$.
These Casimir invariants are  fundamental in the study of nonlinear stability of 2D flows, 
and in particular, were  the basis for Arnold's stability criterion in ideal hydrodynamics, see \cite{Arn66, AK}. 
In the energy-Casimir method one studies the second variation of the energy functional 
with an appropriately chosen combination of Casimirs. 

In the case of a flow in an annulus with a vorticity function without critical points such invariants form a complete set 
of Casimirs \cite{ChSv}, while for more complicated functions and domains it is not so. 
In Appendix B we give a complete description of Casimirs in the general setting of Morse vorticity functions 
on two-dimensional surfaces with or without boundary.

%%%%%%%%%%%%%%%%%%%%%%%%%%%%%%

\section{Coadjoint orbits of the symplectomorphisms group}\label{sdiffOrbitsSection}

\subsection{Simple Morse functions and measured Reeb graphs}
\begin{definition}
Let $M$ be a compact connected surface with a possibly non-empty boundary $\partial M$.
A Morse function  $F \colon M \to \R $ is called \textit{simple} if it satisfies the following conditions:

\begin{longenum}
\item $F$ is constant on each connected component of the boundary $\partial M$;
\item $F$ does not have critical points at the boundary;
\item any $F$-level\footnote{In what follows, by an $F$-level we mean a connected component of the set $F = \const$.} contains at most one critical point.\end{longenum}
\end{definition}
With each simple Morse function $F \colon M \to \R$, one can associate a graph. This graph $\Gamma_F$ is defined as the space of $F$-levels with 
the induced quotient topology. Each vertex of this graph
corresponds either to a critical level of the function $F$, or to a boundary component of the surface $M$.

The function $F$ on $M$ descends to a function $f$ on the graph $ \Gamma_F$. %In what follows, by a Reeb graph we always mean a pair: a graph, and a function on it.  
It is also convenient to assume that $\Gamma_F$ is oriented: edges are oriented in the direction of increasing $f$. \par
To distinguish between vertices of  $ \Gamma_F$ corresponding to critical levels of $F$ and vertices corresponding to boundary components, we denote the latter by circles and call them \textit{boundary vertices}. We shall use the notation $\partial\Gamma$ for the set of boundary vertices.

 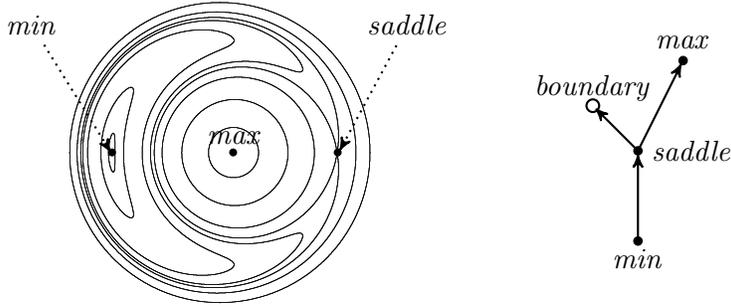
\begin{figure}[t]
 \centering
\begin{tikzpicture}[thick]
\node () at (0,0) {\includegraphics[scale = 1]{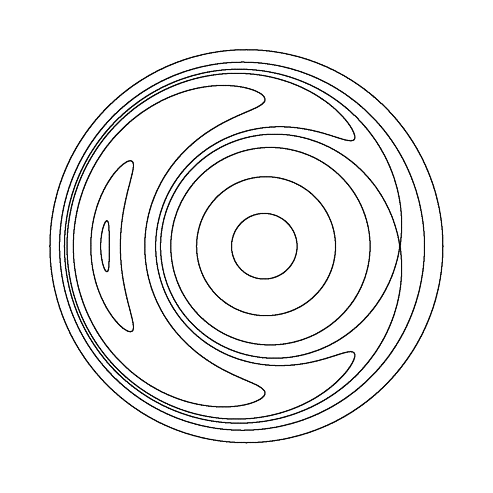}};
\node [] () at (0.2,0.2) {\textit{max}};
 \fill (0.18,0) circle [radius=1.5pt];
\node [] (U) at (-2.5,1.7) {\textit{min}};
\draw [->, dotted]  (U) -- (-1.45,0.01);
 \fill (-1.43,0.0) circle [radius=1.5pt];
  \fill (1.57, 0) circle [radius=1.5pt];
  \node [] (U2) at (2.5,1.7) {\textit{saddle}};
\draw [->, dotted]  (U2) -- (1.6,0.03);
\node () at (5.5, 0)
{
\begin{tikzpicture}[thick, scale = 1.2]
\node [vertex] (A) at (5,0.5) {};
\node [] () at (5,0.7) {\textit{boundary}};
\node[vertex] (B) at (5.5,0) {};
\node [] () at (6.1,0) {\textit{saddle}};
\node[vertex] (C) at (6,1) {};
\node [] () at (6,1.2) {\textit{max}};
\node[vertex] (D) at (5.5,-1) {};
\node [] () at (5.5,-1.2) {\textit{min}};
    \draw (A) circle [radius=2pt];
    \fill (B) circle [radius=1.5pt];
        \fill (C) circle [radius=1.5pt];
            \fill (D) circle [radius=1.5pt];
    \draw [a] (B) -- (A);
        \draw [a] (B) -- (C);
            \draw [a] (D) -- (B);            
            \end{tikzpicture}
            };
\end{tikzpicture}
\caption{A simple Morse function on a disk and the associated graph.}\label{function}
\end{figure}

\begin{example}
{\rm
Figure \ref{function} shows level curves of a simple Morse function on a disk and the corresponding graph $\Gamma_F$.
}
\end{example}

\begin{definition}\label{RGDef}
A \textit{Reeb graph} $(\Gamma,  f)$ is an oriented connected graph  $\Gamma$ with a continuous function 
$f \colon \Gamma \to \R$  which satisfy the following properties.
\begin{longenum}
\item All vertices of $\Gamma$ are either $1$-valent or $3$-valent. 
\item $1$-valent vertices are of two types: boundary vertices and non-boundary vertices.
\item For each $3$-valent vertex, there are either two incoming and one outgoing edge, or vice versa. \end{longenum}

\end{definition}
It is a standard result from Morse theory that the graph $\Gamma_F$ associated with a simple Morse function 
$F \colon M \to \R$ on an orientable connected surface $M$ is a Reeb graph in the sense of Definition \ref{RGDef}. 
We will call this graph the {Reeb graph} of the function $F$. Note that Reeb graphs classify simple Morse functions on $M$ up to diffeomorphisms.
\par
In what follows, we assume that the surface $M$ is endowed with an area (i.e., symplectic) form $\omega$. 
We are interested in the classification problem for simple Morse functions up to area-preserving (i.e., symplectic) 
diffeomorphisms. It turns out that this classification can be given in terms of so-called \textit{log-smooth measures} 
on Reeb graphs.
  \begin{figure}[t]
\centerline{
\begin{tikzpicture}[thick]
\draw  [->] (2,0) -- (2,1.5);
\draw  [->] (2,1.5) -- (3.5,3);
\draw  [->] (2,1.5) -- (0.5,3);
\draw [decoration={text along path,
    text={trunk},text align={center},  raise = 0.1 cm},decorate]  (2,0) -- (2,1.5);
    \draw [decoration={text along path,
    text={branch},text align={center},  raise = -0.3 cm},decorate] (2,1.5) -- (3.5,3);
        \draw [decoration={text along path,
    text={branch},text align={center}, raise = -0.3cm},decorate]   (0.5,3) -- (2,1.5);
    \draw  [->] (5,0) -- (6.5,1.5);
       \draw  [->] (8,0) -- (6.5,1.5); 
              \draw  [->] (6.5,1.5) -- (6.5,3); 
              \draw [decoration={text along path,
    text={trunk},text align={center},  raise = 0.1 cm},decorate] (6.5,1.5) -- (6.5,3);
    \draw [decoration={text along path,
    text={branch},text align={center},  raise = 0.1 cm},decorate]  (5,0) -- (6.5,1.5);
        \draw [decoration={text along path,
    text={branch},text align={center}, raise = 0.1cm},decorate]    (6.5,1.5) -- (8,0); 
\end{tikzpicture}
}
\caption{Trunk and branches in a Reeb graph.}\label{trb}
\end{figure}
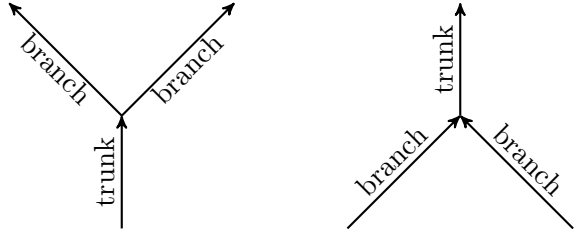
 \begin{definition}
 Let $\Gamma$ be a Reeb graph. Assume that $e_0, e_1$, and $e_2$ are three edges of  $\Gamma$ which meet at a $3$-valent vertex $v$. Then $e_0$ is called the \textit{trunk} of $v$, and $ e_1, e_2$ are called \textit{branches} of $v$ if either $e_0$ is an outgoing edge for $v$, and $e_1, e_2$ are its incoming edges, or  $e_0$ is an incoming edge for $v$, and $e_1, e_2$ are its outgoing edges (see Figure \ref{trb}).
\end{definition}
\begin{definition}\label{measureproperty} A measure $\mu$ on a Reeb graph $(\Gamma,f)$ is called \textit{log-smooth} if it has the following properties:
\begin{longenum}
\item It has a $C^\infty$ smooth non-vanishing density $\diff \mu / \diff f$ at interior points and $1$-valent vertices of $\Gamma$.
\item At $3$-valent vertices, the measure $\mu$ has logarithmic singularities. More precisely, one has the following. Assume that $v$ is a $3$-valent vertex of $\Gamma$. Without loss of generality assume that $f(v)=0$ (if not, we replace $f$ by $\tilde f(x):=f(x)-f(v)$).
Let $e_0$ be the trunk of $v$, and let $e_1, e_2$ be the branches of $v$.
 Then there exist functions $ \psi, \eta_0, \eta_1,\eta_2$ of one variable, smooth in the neighborhood of the origin $0\in \R$ and such that for any point $x \in e_i$ sufficiently close to $v$, we have
\begin{align*}
\mu([v,x])&= \,\eps_i\psi(  f(x) )\ln |   f (x)| + \eta_i(   f(x) ),
\end{align*}
where $\eps_0 = 2$, $\eps_1 = \eps_2 =  -1$,  $\psi(0) = 0$, $\psi'(0) \neq 0$,
and $\eta_0 + \eta_1 + \eta_2 = 0$.
\end{longenum}
\end{definition}
\begin{definition}
 A Reeb graph $(\Gamma,  f)$ endowed with a log-smooth measure $\mu$ is called a \textit{measured Reeb graph}.
 \end{definition}
  
  If a surface $M$ is endowed with an area form $\omega$, then the Reeb graph $\Gamma_F$ of any simple Morse function $F \colon M \to \R$ has a natural structure of a measured Reeb graph. The measure $\mu$ on $\Gamma_F$ is defined as the pushforward of the area form on $M$ under the natural projection $\pi \colon M \to \Gamma_F$.\par
  It turns out that there is a one-to-one correspondence between simple Morse functions on $M$, considered up to symplectomorphisms, and measured Reeb graphs satisfying the following natural compatibility conditions:

\begin{definition}\label{def:compatible}
{\rm
Let $M$ be a connected surface, possibly with boundary, endowed with a symplectic form $\omega$.
A measured Reeb graph  $(\Gamma,  f, \mu)$ is \textit{compatible with} $(M, \omega)$ if
\begin{longenum}
\item the genus $\dim \Hom_1(\Gamma, \R)$ of $\Gamma$ is equal to the genus 
of  $M$;
\item the number of boundary vertices of $\Gamma$ is equal to the number of components of $\partial M$;
\item
 the volume of $\Gamma$ with respect to the measure  $\mu$ 
is equal to the volume of $M$: $\int_\Gamma \diff \mu=\int_M\omega.$
\end{longenum}
}
\end{definition}

\begin{theorem}\label{thm:sdiff-functions}
The mapping  assigning the measured Reeb graph $\Gamma_F$ to a simple Morse function $F$ provides a one-to-one correspondence between simple Morse functions on $M$ up to a symplectomorphism 
and measured Reeb graphs compatible with $M$.
\end{theorem}

When $M$ is closed, this result is Theorem 3.11 of \cite{IKM}, and, as we explain below, the general case can be reduced to the case of closed $M$ by means of the symplectic cut operation.
\par

 Let $M$ be a connected surface with boundary endowed with a symplectic form $\omega$ and let  $F \colon M \to \R$ 
 be a simple Morse function. Consider the topological space $\closure M$ obtained from $M$ by contracting each 
 connected component of $\partial M$ into a point. Denote by  $\pi \colon M \to \closure M$  the natural projection, and by the 
  $\closure F$ and $\closure \omega$ the pushforwards to the $\closure M$ of the corresponding function $ F$ 
  and symplectic structure $\omega$.
  
 \begin{proposition}
 There exists a smooth structure on $\closure M$ such that the function $F$ and symplectic structure $\omega$ descend to smooth objects  $\closure F$ and $\closure \omega$ on $\closure M$. 
  \end{proposition}

 \begin{proof}
In a sense, this is a nonlinear version of the Archimedes theorem that the radial projection from a sphere 
to the circumscribed cylinder is area-preserving. Indeed, in a neighborhood  of every boundary component $\ell$ of $M$
introduce  \textit{action-angle coordinates}, i.e. smooth functions 
$$
S \colon U \to \R_{\geq 0} = \{ z \in \R \mid z \geq 0\}, \quad \Theta \colon U \to S^1 = \R \, / \, 2\pi\Z,
$$ 
such that $\omega = \diff S\wedge \diff \Theta$, $S = 0$ on $\ell$, and $F = \zeta(S)$ where $\zeta$ is a smooth function 
in one variable such that $\zeta'(0) \neq 0$.  Then in a neighborhood of the point $Z = \pi(\ell)\in \closure M$, the
 image of this boundary component,  we define a chart by taking the functions 
 $P := \sqrt{S}\cos \Theta$, $Q := \sqrt{S}\sin \Theta$  as local coordinates.
Then  the function  $F$ in terms of functions $P$, $Q$ near each point $Z$ becomes $\closure F = \zeta(P^2 + Q^2)$,
while  the symplectic structure now is $\closure \omega = 2 \diff P \wedge \diff Q$.
 \end{proof}

  \begin{proof}[Proof of  Theorem \ref{thm:sdiff-functions}]
We have a functor $   (M,  F,  \omega) \mapsto (\closure M, \closure F, \closure \omega)$ from the category of triples 
$( M,  F,  \omega)$ on surfaces with boundary to the category of triples $(\closure M, \closure F, \closure \omega)$ 
on closed surfaces. Note that in order to invert this functor one needs to specify minima and maxima of the function 
$\closure F$ that correspond to boundary components of $M$. At the level of Reeb graphs, this correspondence 
replaces boundary vertices by marked vertices. 
Therefore, we have a one-to-one correspondence between triples $( M,  F,  \omega)$ (with boundary) and triples 
$(\closure M, \closure F, \closure \omega)$ (without boundary) with marked minima and maxima. 
Since measured Reeb graphs 
up to isomorphism completely describe simple Morse functions on closed symplectic surfaces 
 modulo symplectomorphisms (see Theorem~3.11 of~\cite{IKM}),  
this one-to-one correspondence extends to symplectic surfaces with boundary.
\end{proof}

  \begin{remark}
Note that the functor $   (M,  F,  \omega) \mapsto (\closure M, \closure F, \closure \omega)$ may be viewed as a particular case of the symplectic cut operation introduced by E.\,Lerman \cite{lerman1995symplectic}.
 \end{remark}

%%%%%%%%%%%%%%%%%%%%%%%%%%%%%%%%%%%%%%%%%%% 
 
 \medskip
 \subsection{Antiderivatives on graphs}\label{sect:antider}

 In order to pass from the above classification of simple Morse functions on symplectic surfaces to the classification of coadjoint orbits of the group $\SDiff(M)$, we need to introduce the notion of the antiderivative of a density on a graph. Let $\Gamma$ be an oriented graph, some of whose vertices are marked and called boundary vertices. Let also $ \rho$ be a \textit{density} on $\Gamma$, i.e. a finite Borel signed measure.
 
 \begin{definition}\label{circProperties}
A function $\lambda \colon \Gamma \,\setminus\, V \to \R$ defined and continuous on the graph $\Gamma$ outside its set of vertices $V = V(\Gamma)$ is called an \textit{antiderivative of the density $\rho$} if it has the following properties.
\begin{longenum}
\item It has at worst jump discontinuities at vertices, which means that for any vertex $v \in V$ and any edge $e \ni v$, there exists a finite limit
$$
\lim\nolimits_{{x \xrightarrow[]{e} v }} \lambda(x),
$$
where $ {{x \xrightarrow[]{e} v }} $ means ``as $x$ tends to $v$ along the edge $e$''.
\item Assume that $x,y$ are two interior points of some edge $e \in \Gamma$, and that $e$ is pointing from $x$ towards $y$. Then $\lambda$ satisfies the Newton-Leibniz formula
\begin{align}
\label{stokes}
\lambda(y) - \lambda(x) = \rho([x,y]).
\end{align}
%where $[x,y]$ is the part of the edge $e$ enclosed by $x$ and $y$.
%\item Let $v$ be a non-boundary $1$-valent vertex of $\Gamma_F$. Then
%\begin{align}
%\label{1valentcirc}
%\lim\limits_{x \to v} \circulation(x) = 0\,.
%\end{align}
\item Let $v$ be a non-boundary vertex of $\Gamma$. Then $\lambda$ satisfies the Kirchhoff rule at $v$:
%Let $e_0$ be the trunk of $v$, and let $e_1, e_2$ be the branches of $v$. Let also $x_i \in e_i$. Then
\begin{align}
\label{3valentcirc}
 \sum_{{e \to v}} \lim\nolimits_{{x \xrightarrow[]{e} v }} \lambda(x)= \sum_{{e \leftarrow v}} \lim\nolimits_{{x \xrightarrow[]{e} v }} \lambda(x)\,,
\end{align}
where the notation $e \to v$ stands for the set of edges pointing at the vertex $v$, and ${e \leftarrow v}$ stands for the set of edges pointing away from $v$. 
\end{longenum}
\end{definition}
%\begin{definition}
%A measured Reeb graph endowed with a circulation function is called a \textit{circulation graph}. 
%\end{definition}

\begin{proposition}\label{circFuncs}
Let $\Gamma$ be an oriented graph with marked boundary vertices. 
\begin{longenum}
\item If the graph $\Gamma$ has at least one boundary vertex, then any density $\rho$ on $\Gamma$ admits an antiderivative.
\item 
If $\Gamma$ has no boundary vertices, then a density $\rho$ on $\Gamma$ admits an antiderivative if and only if
$
\rho(\Gamma) = 0
$.
\item
If a density $\rho$ on $\Gamma$ admits an antiderivative, then the set of antiderivatives of $\rho$ is an affine space whose underlying vector space is the relative homology group $\Hom_1(\Gamma, \partial\Gamma ,\R)$, where $\partial \Gamma$ is the set of boundary vertices of $\Gamma$.
\end{longenum}
\end{proposition}
 \begin{figure}[t]
\centerline{
\begin{tikzpicture}[thick]
    \node [vertex] at (7,-0.2) (nodeC) {};
    \node [vertex]  at (7,1.05) (nodeD) {};
    \node [vertex] at (7,2.75) (nodeE) {};
    \node [vertex]  at (7,4.2) (nodeF) {};
    \draw  [a] (nodeC) -- (nodeD);
    \fill (nodeC) circle [radius=1.5pt];
    \fill (nodeD) circle [radius=1.5pt];
    \fill (nodeE) circle [radius=1.5pt];
    \fill (nodeF) circle [radius=1.5pt];
  \draw [a] (nodeD) .. controls +(0.7, 0.5) and +(0.7,-0.5) .. (nodeE);
    \draw [a] (nodeD) .. controls +(-0.7, 0.5) and +(-0.7,-0.5) .. (nodeE);
    \draw  [a] (nodeE) -- (nodeF);
    \node at (7.8,1.9) () {{${e_3}$}};
      \node at (6.2,1.9) () {{${e_2}$}};
     \node at (7.3,3.5) () {{${e_4}$}};
          \node at (6.8,0.3) () {{${e_1}$}};
          \node at (7.2, -0.1) () {$0$};
                \node at (6.8, 4.1) () {$0$};
                    \node at (7.3, 0.85) () {$a_1$};
                            \node at (7.8, 1.2) () {$a_1 - z$};
                              \node at (6.6, 1.2) () {$z$};
                                   \node at (6.15, 2.6) () {$ a_2 + z$};
                                   \node at (8.25, 2.6) () {$a_1 + a_3 -z$};
                                    \node at (6.65, 3) () {$-a_4$};
\end{tikzpicture}
}
\caption{The space of antiderivatives on a graph of genus one.}\label{circTorus}
\end{figure}
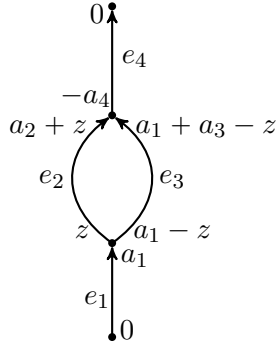

\begin{example}\label{example:circtorus}

Consider the graph $\Gamma$ depicted in Figure \ref{circTorus}.  Let $\rho$ be a density on $\Gamma$ such that
$\rho(e_i) = a_i$, where the numbers $a_i$ satisfy $a_1 + a_2 + a_3 + a_4 = 0$ (so that the density $\rho$ admits an antiderivative). The numbers near vertices in the figure stand for the limits of the antiderivative $\lambda$ of $\rho$.  The space of such antiderivatives has one parameter
$z$, which determines how the splitting of the antiderivative between the edges $e_2$ and $e_3$ takes place  
(by the proposition above the space of antiderivatives is one-dimensional).
% Given a function $f$ the corresponding circulation  (or antiderivative) function $\circulation$ on the graph has one parameter
% $z$, which determines how the splitting of the  antiderivative at $v$ takes place  
% (by the proposition above the space of circulation functions on $\Gamma$ is one-dimensional).
%Also note that $a_1 + a_2 + a_3 + a_4 = 0$ (otherwise, the graph $\Gamma$ does not admit a circulation function and does not correspond to any coadjoint orbit of $\SDiff(T^2)$, according to Proposition~\ref{circFuncs}). 
\end{example}
\begin{proof}[Proof of Proposition \ref{circFuncs}]
Let $\lambda \colon \Gamma \, \setminus \, V \to \R$ be any function satisfying conditions i) and ii) of Definition \ref{circProperties}. Let $e$ be an edge of $\Gamma$ going from $v$ to $w$, and let $x \in e$. Denote by $\lambda^-(e)$ and $\lambda^+(e)$ the limits of $\lambda(x)$ as $x$ tends to $v$ and $w$ respectively. 
Consider the $1$-chain
\begin{align}\label{1chain}
\mathrm{ch}(\lambda) := \sum_{{e\, \in \, E}} \lambda^+(e)e\,,
\end{align}
where $E = E(\Gamma)$ stands for the set of edges of $\Gamma$. Note that 
the mapping $\mathrm{ch}$ is a one-to-one map between functions $\lambda \colon \Gamma \, \setminus \, V \to \R$ satisfying conditions i) and ii) of Definition \ref{circProperties} and $1$-chains on $\Gamma$. Then we have
\begin{align*}
\partial \mathrm{ch}(\lambda) =\sum_{{v \, \in \, V}}\left( \sum_{{e \to v}} \lambda^+(e) - \sum_{{e\leftarrow v}}  \lambda^+(e) \right)v = \chi + \sum_{{v \, \in \, V}}\Delta(v) v \,,
\end{align*}
where
$\chi$ is a $0$-chain given by
$$
\chi := -\sum_{{v \, \in \, V}}\sum_{{e \leftarrow v}} \rho(e)v\,,
$$
and
$$
\Delta(v) :=  \sum_{{e \to v}} \lambda^+(e) \,\,- \sum_{{e \leftarrow v}}  \lambda^-(e) \,.
$$
Note that $\Delta(v) = 0$ is equivalent to the Kirchhoff rule at $v$. Therefore, $\lambda$ is an antiderivative of $\rho$ if and only if
\begin{align}
\label{boundaryCond}
\partial \mathrm{ch}(\lambda) = \chi \mod \partial\Gamma.
\end{align}
So, a density $\rho$ on $\Gamma$ admits an antiderivative if and only if the $0$-chain $\chi$ is a relative boundary modulo $\partial\Gamma$. If the set $\partial\Gamma$ is non-empty, then we have $\Hom_0(\Gamma, \partial\Gamma ,\R) = 0$, so the chain $\chi$ is always a relative boundary, which proves the first statement of the proposition. Further, if $\partial\Gamma$ is empty, then $\chi$ is a boundary if and only if $\eps(\chi) = 0$ where $\eps$ is the augmentation map $\eps\left(\sum c_i v_i\right) = \sum c_i$. We have
$
\eps(\chi) = -\rho(\Gamma),
$
which proves the second statement.  To prove the third statement, note that since the mapping $\mathrm{ch}$ identifies
appropriate functions   $\lambda \colon \Gamma \, \setminus \, V \to \R$  and $1$-chains on $\Gamma$, it  
 identifies antiderivatives of $\rho$ with the set of $1$-chains having property~\eqref{boundaryCond}. Now, to complete the proof, it suffices to note that the latter set is an affine space with underlying vector space $\Hom_1(\Gamma, \partial\Gamma ,\R)$.
\end{proof}

\medskip

\subsection{Classification of coadjoint orbits}\label{sect:coadj-sdiff}
Let $M$ be a compact connected surface, possibly with boundary. Assume that $M$ is endowed with a symplectic 
form $\omega$.
Recall that the regular dual $\SVect^*(M)$ of the Lie algebra $\SVect(M)$ of divergence-free
vector fields on a surface $M$  is identified with the space $\Omega^1(M) / \diff \Omega^0(M)$ 
of smooth $1$-forms modulo exact $1$-forms on $M$. 
The coadjoint action of a $\SDiff(M)$ on $\SVect^*(M)$ is given by the change of coordinates in (cosets of) 1-forms on $M$ by means of a symplectic diffeomorphism:
$
\Ad^*_\Phi \,[\oneform] = [\Phi^*\oneform].
$
In what follows, the notation $[\alpha]$ stands for the coset of 1-forms 
$\alpha$ in $\Omega^1(M) / \diff \Omega^0(M)$. In particular, if the form $\alpha$ is closed, then $[\alpha]$ is the cohomology class of $\alpha$.
\par
To describe orbits of the coadjoint action of $\SDiff(M)$ on $\SVect^*(M) $, consider the mapping
$
\mathfrak{curl} \colon \Omega^1(M) \,/\, \diff \Omega^0(M) \to C^\infty(M)
$
given by taking the vorticity function
$$
\mathfrak{curl}[\oneform] := \frac{\diff \oneform}{\omega}.
$$
(One can view this map as taking the vorticity  of a vector field $u=\alpha^\sharp$.)\begin{remark}\label{imageOfD}
If the boundary of $M$ is non-empty, then the mapping $\mathfrak{curl}$ is surjective. Otherwise, its image is the space of functions with zero mean.
\end{remark}
An important property of the mapping $\mathfrak{curl}$ is its equivariance with respect to the $\SDiff(M)$ action: if cosets $[\oneform], [\oneformtwo] \in \SVect^*(M)$ belong to the same coadjoint orbit, then the functions $\mathfrak{curl}[\oneform]$ and $  \mathfrak{curl}[\oneformtwo]$ are conjugated by a symplectic diffeomorphism. In particular, if  $\mathfrak{curl}[\oneform]$ is a simple Morse function, then  so is $  \mathfrak{curl}[\oneformtwo]$.

\begin{definition}
{\rm
We say that a coset of 1-forms $[\oneform] \in \SVect^*(M)$ is \textit{Morse-type} if $\mathfrak{curl}[\oneform]$ is a simple Morse function.
A coadjoint orbit $\orbit \subset \SVect^*(M)$ is \textit{Morse-type} if any coset $[\oneform] \in \orbit$ is Morse-type (equivalently, if at least one coset $[\oneform] \in \orbit$ is Morse-type).
}
\end{definition}
\par
Let $[\oneform] \in \SVect^*(M)$ be Morse-type, and let $F = \mathfrak{curl}[\oneform]$.   
Consider the measured Reeb graph $\Gamma_F$. Since $\mathfrak{curl}$ is an equivariant mapping, this graph is invariant under the coadjoint action of $\SDiff(M)$ on $\SVect^*(M)$. However, this invariant is not complete if $M$ is not simply connected (i.e., if $M$ is neither a sphere $S^2$, nor a disk $D^2$). To construct a complete invariant, we  endow the graph $\Gamma_F$ with a \textit{circulation function} constructed as follows. Let $\pi \colon M \to \Gamma_F$ 
be the natural projection. Take any point $x$ lying in the interior of some edge 
$e \in \Gamma_F$. Then $\pi^{-1}(x)$ is a circle $\ell_x$. It is naturally oriented 
as the boundary of the set of smaller values. The integral of $\oneform$ 
over $\ell_x$ does not depend on the choice of a representative $\oneform \in [\oneform]$.
Thus, we obtain a function 
$
\circulation \colon \Gamma_F\, \setminus \, V( \Gamma_F) \to \R
$
given by 
\begin{align}\label{cosetCirculation}
\circulation(x) := \int\nolimits_{\ell_x} \!\!\oneform\,.
\end{align}
Note that in the presence of a metric on $M$, the value $\circulation(x)$ is the circulation over the level 
$\pi^{-1}(x)$ of the vector field  $\alpha^\sharp$ dual to the 1-form $ \oneform$.
\begin{proposition}\label{circIsAntiDer}
For any Morse-type coset $[\alpha] \in \SVect^*(M)$, the function $\circulation$ given by formula~\eqref{cosetCirculation} is an antiderivative of the density $\rho(I) := \int_I f \diff \mu$ in the sense of Definition \ref{circProperties}.
\end{proposition}
\begin{remark}
This density $\rho$ is the pushforward of the vorticity $2$-form $\diff[\alpha]$ from the surface to the Reeb graph.
\end{remark}
\begin{proof}[Proof of Proposition \ref{circIsAntiDer}]
The proof is straightforward and follows from the Stokes formula and additivity of the circulation integral.
\end{proof}
\begin{definition}
Let $(\Gamma, f, \mu)$ be a measured Reeb graph. A \textit{circulation function} $\circulation$ on $\Gamma$ is an antiderivative of the density $\rho(I) := \int_I f \diff \mu$.
%\begin{definition}
A measured Reeb graph endowed with a circulation function is called a \textit{circulation graph}. 
%\end{definition}
\end{definition}
So, with any Morse-type coset $[\alpha] \in \SVect^*(M)$ we associate a circulation graph. Denote this graph by $\Gamma_{[\alpha]}$.
\begin{theorem}\label{thm4} \label{thm:sdiffM}
Let $M$ be a compact connected symplectic  surface. Then Morse-type coadjoint orbits of $\SDiff(M)$ are in one-to-one correspondence with circulation graphs  
$(\Gamma, f, \mu, \circulation)$ compatible with $M$.
In other words, the following statements hold:
\begin{longenum}
\item For a symplectic surface $M$  Morse-type cosets $[\oneform], [\oneformtwo] \in \SVect^*(M)$ 
 lie in the same orbit of the $\SDiff(M)$ coadjoint action if and only if circulation graphs $\Gamma_{[\oneform]}$ and $\Gamma_{[\oneformtwo]}$ corresponding to these cosets are isomorphic.
\item For each circulation graph $\Gamma$ which is compatible\footnote{See Definition \ref{def:compatible} for compatibility of a graph and a surface.}
 with $M$, there exists a Morse-type $[\oneform] \in   \SVect^*(M)$ such that  $\Gamma_{[\oneform]} =(\Gamma, f, \mu, \circulation)$.
\end{longenum}
\end{theorem}
We start with the following preliminary lemma.

\begin{lemma}\label{graphCohomology}
Let $M$ a connected oriented surface possibly with boundary, and let $F$ be a simple Morse function on $M$. 
Assume that $[\gamma] \in  \Hom^1(M,\R)$ is such that the integral of $\gamma$ over 
any $F$-level vanishes. Then there exists a $C^\infty$ 
function $H \colon M \to \R$ constant on the boundary $\partial M$ such that the $1$-form $H\diff F$ is closed, and its cohomology class 
is equal to $[\gamma]$. Moreover, $H$ can be chosen in such a way that the ratio
$H/F$ is a smooth function.
\end{lemma}
\begin{proof}
Since the integral of $[\gamma]$ over any connected component of any $F$-level vanishes, 
the cohomology class $[\gamma]$ on $M$ belongs to the image of the inclusion
$
i \colon \Hom^1(\Gamma_F, \R) \to \Hom^1(M, \R).
$
Let $\alpha$ be a $1$-cochain, i.e. a real-valued function on edges, on the graph $\Gamma_F$ representing the cohomology class $i^{-1}[\gamma]$. Recall that the graph $\Gamma_F$ is endowed with a function $f$ that is the pushforward  of the function $F$.
Consider a continuous function $ h \colon \Gamma_F \to \R$ such that 
\begin{longenum}
\item it is a smooth function of $f$ in a neighborhood of each point $x \in \Gamma_F$;
\item it vanishes if $f$ is sufficiently close to zero;
\item for each edge $e$, we have
$
\alpha (e) = \int_e  h\diff f\,.
$
\end{longenum}
Obviously, such  a function does exist. Now, lifting $ h$ to $M$, we obtain a smooth function $H$ with 
the desired properties.
\end{proof}
\begin{proof}[Proof of Theorem \ref{thm4}]
For the first statement, it is immediate by construction, that  if cosets $[\oneform]$ and $ [\oneformtwo]$ 
lie in the same $\SDiff(M)$-orbit then their circulation graphs $\Gamma_{[\oneform]}$ and $\Gamma_{[\oneformtwo]}$ 
are isomorphic.
To prove the converse statement let $\phi \colon \Gamma_{[\oneform]} \to \Gamma_{[\oneformtwo]}$ be an isomorphism 
of circulation graphs. By Theorem \ref{thm:sdiff-functions}, $\phi$ can be lifted to a symplectomorphism 
$ \Phi \colon M \to M$ that  maps the function $F = \mathfrak{curl}[\oneform]$ 
to the function $G = \mathfrak{curl}[\oneformtwo]$. Therefore, the $1$-form $\oneformthree$ defined by
$$
\oneformthree :=  \Phi^*\oneformtwo - \oneform 
$$
is closed. 
\par
Assume that $ \Psi \colon M \to M $ is a symplectomorphism which maps the function $F$ to itself and is isotopic to the identity. Then the composition $\widetilde\Phi = \Phi\circ\Psi^{-1}$ maps $F$ to $G$, and
$$
[\widetilde\Phi^*\oneformtwo - \oneform] = [\Phi^*\oneformtwo - \Psi^*\oneform] = [\gamma] - [\Psi^*\oneform - \oneform].
$$
We claim that $\Psi$ can be chosen in such a way that $\widetilde\Phi^*\oneformtwo - \oneform$ is exact, i.e.  one has the equality of the cohomology classes
$
 [\Psi^*\oneform - \oneform] = [\gamma].
$
We construct $\Psi$ using a version of the Moser path method. Let us show that there exists a time-independent symplectic vector field  $X$ which is tangent to the boundary $\partial M$,
 preserves $F$, and satisfies
\begin{align}\label{homotopy2}
 [\Psi_t^*\oneform - \oneform] = t[\gamma]\,,
\end{align}
where $\Psi_t$ is the phase flow of $X$. 
Differentiating \eqref{homotopy2} with respect to $t$, we get in the left-hand side
$$
 [\Psi_t^*L_X\oneform] =  [L_X\oneform] =  [i_X\diff \oneform] = [F\cdot i_X \omega]\,,
$$
since $L_X \oneform$ is closed and $\Psi_t^*$ does not change its cohomology class. Thus
\begin{align}\label{homology2}
[F\cdot i_X \omega] = [\gamma]\,.
\end{align}
Since $\Phi$ preserves the circulation function, the integrals of $\oneformthree$ over all $F$-levels vanish. Therefore, by Lemma \ref{graphCohomology}, there exists a smooth function $H$ such that
$
[\gamma] = [H\diff F].
$
Now we set 
$$
X :=  \frac{H}{F} \,\omega^{-1}\diff F\,.
$$
It is easy to see that the vector field $X$ is symplectic, preserves the levels of $F$, and satisfies the equation \nolinebreak \eqref{homology2}. Therefore, its phase flow satisfies the equation \eqref{homotopy2}, and then for the symplectomorphism $\widetilde\Phi = \Phi \circ \Psi_1^{-1}$, where $\Psi$ is the time-one map $\Psi_1$, we have $\widetilde\Phi^*[\beta] = \alpha$, as desired.
\par
Now, let us prove the second statement. By Theorem \ref{thm:sdiff-functions}, there exists a simple Morse function $F \colon M \to \R$ such that the measured Reeb graph of $F$ is $(\Gamma, f, \mu)$. 
Since the graph $\Gamma$ admits a circulation function, from Proposition \ref{circFuncs} and Remark \ref{imageOfD} it follows that the function $F$ lies in the image of the mapping $\mathfrak{curl}$, i.e. there exists a $1$-form $\alpha \in \Omega^1(M)$ such that $\mathfrak{curl}[\alpha] = F$. 
Further, if $\gamma$ is a closed $1$-form, then $\mathfrak{curl}[\alpha + \gamma] = F$ as well. 
For any $1$-form $\tilde\alpha$ such that  $\mathfrak{curl}[\tilde\alpha] = F$, let $\circulation_{\tilde\alpha}$ denote the corresponding circulation function on $\Gamma$. Consider the mapping
$
\rho \colon \Hom^1(M,\R) \to \Hom_1(\Gamma,\partial\Gamma, \R) 
$
given by
$$
\rho[\gamma] := \mathrm{ch}(\circulation_{\alpha + \gamma}) - \mathrm{ch}(\circulation_{\alpha})\,,
$$
where $\mathrm{ch}(\circulation)$ is given by \eqref{1chain}. Let us show that the homomorphism $\rho$ is surjective. Indeed, $\rho$ can be written as
$$
\rho[\gamma] = \sum_{e} \left(\int_{{\ell(e)}}\!\!\gamma\right) e\,,
$$
where $\ell(e) = \pi^{-1}(x_e)$ and $x_e \in e$ is any interior point of the edge $e$. 
Therefore, the kernel of the homomorphism $\rho$ consists of those cohomology classes 
which vanish on cycles homologous to regular $F$-levels, 
and hence $\dim \Ker \rho = \varkappa$, where $\varkappa$ is the genus of the surface $M$. On the other hand, comparing the dimensions of the spaces involved, we have
\begin{align}\label{dimH1M}
\dim  \Hom^1(M,\R) = \begin{cases}
2\varkappa, \mbox{ if $M$ is closed,}\\
2\varkappa + k -1, \mbox{ if $M$ is not closed}\
\end{cases}
\end{align}
where $k$ is the number of boundary components of $M$ or, equivalently, the number of boundary vertices of $\Gamma$. Further, from the exact sequence of the pair $(\Gamma, \partial\Gamma)$ we find that
\begin{align}\label{dimRelativeHomology}
\dim  \Hom_1(\Gamma,\partial\Gamma, \R) = \begin{cases}
\varkappa, \mbox{ if } k = 0\,,\\
\varkappa + k -1, \mbox{ if } k >0.
\end{cases}
\end{align}

Since $\dim  \Hom^1(M,\R) - \dim  \Hom_1(\Gamma,\partial\Gamma, \R)= \varkappa$, by the dimensional argument the homomorphism $\rho$ is indeed surjective. This implies that one can find a closed 1-form $\gamma$ such that 
$
\rho[\gamma] = \mathrm{ch}(\circulation) - \mathrm{ch}(\circulation_{\alpha})\,,
$
where $\circulation$ is a given circulation function on $\Gamma$, and therefore 
$
 \circulation_{\alpha + \gamma} = \circulation,
$
as desired.
\end{proof}

%%%%%%%%%%%%%%%%%%%%%%%%%%%%%%%%%%%%%

\section{Steady fluid flows and coadjoint orbits}

\subsection{Orbits admitting steady flows}

\begin{definition}
Let $(M, \omega)$ be a compact symplectic surface, possibly with boundary. We say that a coadjoint orbit $\pazocal O \subset \SVect^*(M)$ \textit{admits a steady flow} if there exists a metric $g$ on $M$ compatible with the symplectic structure $\omega$ and such that the corresponding Euler equation has a steady solution on $\pazocal O$.
\end{definition}
Below we give a necessary and sufficient condition for a Morse-type coadjoint orbit to admit a steady flow. Recall that coadjoint orbits of $\SDiff(M)$ are described by measured Reeb graphs endowed with a circulation function. 
 Let $\Gamma$ be a measured Reeb graph, and let $\circulation$ be a circulation function on $\Gamma$. Let also $v \in V(\Gamma)$ be a $3$-valent vertex of $\Gamma$, and let $e_0, e_1, e_2$ be the edges of $\Gamma$ adjacent to $v$. Denote
\begin{align}\label{limitCirculation}
c_i(v) := \lim\nolimits_{{x \xrightarrow[]{e_i} v}}\, \circulation(x).
\end{align}
\begin{definition}
Let $\Gamma$ be a measured Reeb graph. We say that a circulation function  $\circulation$ on $\Gamma$ is \textit{balanced}
if for any $3$-valent vertex $v \in \Gamma$, the numbers $c_0(v)$, $c_1(v)$, $c_2(v)$ 
defined by formula~\eqref{limitCirculation} are non-zero and have the same sign.
\end{definition}
\begin{remark}
Note that if $e_0$ is the trunk of $v$, and $e_1, e_2$ are branches of $v$, then, by the Kirchhoff rule \eqref{3valentcirc}, we have $c_0(v) = c_1(v) + c_2(v)$, so $c_0(v)$, $c_1(v)$, $c_2(v)$ are non-zero and have the same sign if and only if  $c_1(v)$ and $c_2(v)$ are non-zero and have the same sign.
\end{remark}

\begin{theorem}\label{thm:steady}
Let $(M, \omega)$ be a compact  symplectic surface, possibly with boundary, and let  $\pazocal O \subset \SVect^*(M)$ be a Morse-type coadjoint orbit. Then $\pazocal O$ admits a steady flow if and only if the corresponding circulation function is balanced.
\end{theorem}

\begin{corollary}\label{cor:no3val}
If the Reeb graph of the vorticity function does not have  $3$-valent vertices, then the orbit admits a steady flow. \end{corollary}

Note the there are only four such graphs: a segment with no, either, or two marked points. They correspond to 
vorticity functions with one maximum and one minimum on the sphere, disk, or cylinder. All corresponding 
orbits have steady flows for the metric of constant curvature. Once the vorticity function has saddle points, the existence 
conditions become quite non-trivial.

\begin{example}[see Theorem 5.6 of  \cite{GK}]\label{GKEx}
Assume that $M$ is diffeomorphic to a disk $D^2$. Let $\pazocal O \subset \SVect^*(M)$ be a Morse-type coadjoint orbit such that  the corresponding vorticity function $F$ has a local minimum and a local maximum in  $M\, \setminus \, \partial M$. Then, if $F > 0$ in $M\, \setminus \, \partial M$, the orbit $\pazocal O$ does not contain steady solutions of the Euler equation.\par
Indeed, let $(\Gamma, f, \mu, \circulation)$ be the circulation graph corresponding to the orbit $\pazocal O$. Note that the graph $\Gamma$ has exactly one boundary vertex (since $M$ is a disk), at least one non-boundary $1$-valent vertex with an incoming edge (local maximum of the vorticity), and at least one non-boundary $1$-valent vertex with an outgoing edge (local minimum of the vorticity). Using these three conditions, one can easily show that there exists a directed path $\ell$ in $\Gamma$ which joins two non-boundary $1$-valent vertices (see Figure \ref{path1}). 

 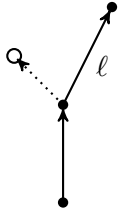
\begin{figure}[t]
 \centering
\begin{tikzpicture}[thick, scale = 1.3]
\node [vertex] (A) at (5,0.5) {};
\node[vertex] (B) at (5.5,0) {};
\node[vertex] (C) at (6,1) {};
\node[vertex] (D) at (5.5,-1) {};
%\node () at (5.65, -0.5) {$\ell$};
\node () at (5.9, 0.4) {$\ell$};
    \draw (A) circle [radius=2pt];
    \fill (B) circle [radius=1.5pt];
        \fill (C) circle [radius=1.5pt];
            \fill (D) circle [radius=1.5pt];
    \draw [a, dotted] (B) -- (A);
        \draw [a] (B) -- (C);
            \draw [a] (D) -- (B);
\end{tikzpicture}
\caption{A directed path $\ell$ joining a local minimum with a local maximum.}\label{path1}
\end{figure}Let us consider the behavior of the circulation function $\circulation$  along the path $\ell$. First, the Newton-Leibniz formula \eqref{stokes} and positivity of the vorticity imply that $\circulation$ is strictly increasing at all points of $\ell$ except, possibly, $3$-valent vertices, where  $\circulation$ has jump discontinuities. Secondly, by the Kirchhoff rule \eqref{3valentcirc}, the function $\circulation$ vanishes at endpoints of $\ell$. Together, these two conditions imply that the function $\circulation$ must change sign at some $3$-valent vertex $v \in \ell$. So, by Theorem~\ref{thm:steady}, the orbit $\pazocal O$ does not admit steady solutions of the Euler equation, q.e.d.\par
Note that the assumption that $M$ is a disk is not important. It suffices to require that the boundary $\partial M$ of $M$ has one connected component.

\end{example}

Below we prove that if a Morse-type coadjoint orbit admits a steady flow, then the corresponding circulation function is balanced. The converse statement is proved in the next section. 
We begin with three preliminary lemmas. The first of them is known as the Morse--Darboux lemma. This lemma is a particular case of {Le lemme de Morse isochore} due to Colin de Verdi{\`e}re and Vey \cite{CDV}.

\begin{lemma}\label{MDL}
Assume that $(M, \omega)$ is a symplectic surface,   and let $F \colon M \to \R$ be a smooth function. 
Let also $O \in M \, \setminus \, \partial M$ be a Morse critical point of $F$. Then, in a neighborhood of $O$, 
there exist coordinates $(P,Q)$ such that  $\omega = \diff P \wedge \diff Q$, and $F = \zeta(S) $ 
where $ S  = \frac{1}{2}(P^2 + Q^2)$ or $ S = PQ$. The function $\zeta$ of one variable is smooth in the neighborhood 
of the origin $0\in\R$, and $\zeta'(0) \neq 0$.
\end{lemma}

Now, let $[\alpha] \in \SVect^*(M)$ be a fixed point of the Euler equation, and let $\alpha \in [\alpha]$ be 
a co-closed representative, a co-closed 1-form $\alpha$ in this coset. The co-closedness condition implies
that the corresponding fluid velocity field $u = \alpha^\sharp$ is divergence-free, i.e. symplectic or locally Hamiltonian, 
and the corresponding locally defined Hamiltonian is called the stream function.
Let also $O \in M$ be a hyperbolic Morse critical point of the vorticity $F = \mathfrak{curl}[\alpha]$. 
Then in a neighborhood of $O$ there exists a stream function $H$, such that $i_u \omega = \diff H$.

\begin{lemma}\label{streamMorse}
Assume that $(M, \omega)$ is a symplectic surface,   and let $[\alpha] \in \SVect^*(M)$ be a fixed point of the Euler equation. 
Let also $O \in M \, \setminus \, \partial M$ be a hyperbolic Morse critical point of the vorticity $F = \mathfrak{curl}[\alpha]$. 
Then $O$ is also a hyperbolic Morse critical point for the stream function $H$. 
Moreover, in coordinates $P, Q$ from Lemma \ref{MDL}, the Taylor expansion of $H$ at $O$ reads
$ H =  a + b PQ + \dots\,,$
where $a,b$ are constants, $b \neq 0$, and dots denote higher-order terms.
\end{lemma}
\begin{proof}
Since $[\alpha]$ is a fixed point for the Euler equation, we have $\{F,H\} = 0$, where $\{\,,\}$ is the Poisson bracket defined by the symplectic structure $\omega$. Applying Lemma \ref{MDL} to  the vorticity $F$ at the  point $O$, we find local coordinates $P, Q$ such that $\omega = \diff P \wedge \diff Q$, and $F = \zeta(PQ)$. Since $\zeta'(0) \neq 0$, we have $PQ = \zeta^{-1}(F)$, and
\begin{align*}
\{PQ, H\} =( \zeta^{-1})'(F) \cdot \{F,H\}= 0\,.
\end{align*}
The latter equation implies that the Taylor expansion of $H$ at $O$ reads
\begin{align}\label{streamFunctionExpansion}
H =  a + b PQ + cP^2Q^2 + \dots\,,
\end{align}
where $a, b,c \in \R$ are constants, and dots denote higher order terms. Assume that $b = 0$. Then
$
H = a + cP^2Q^2 + \dots,
$
and a straightforward computation shows that

\begin{align}\label{laplacian}
\begin{aligned}
F = -\Delta H =  8cg_{12}(O)PQ- 2cg_{11}(O)P^2 - 2cg_{22}(O)Q^2  + \dots\,,
\end{aligned}
\end{align}
where $\Delta$ is the Laplace-Beltrami operator, and $g_{ij}$ are components of the metric $g$ in $P,Q$ coordinates. 
On the other hand, we have
\begin{align*}
F = \zeta(PQ) =\zeta(0) + \zeta'(0)PQ + \dots\,.
\end{align*}
Comparing the latter formula with \eqref{laplacian}, we conclude that $c = 0$, so $\zeta'(0) = 0$, which is a contradiction. Therefore, our assumption is false, and we have $H = a + bPQ +\dots$, where $b \neq 0$, q.e.d.
\end{proof}

Now, recall that if $[\alpha] \in  \SVect^*(M)$ is a steady solution of the Euler equation, then the vorticity function $F = \mathfrak{curl}[\alpha]$ is a first integral for the fluid velocity field $u$. In other words, the fluid velocity field $u$ of a steady flow is tangent to every level of the vorticity $F$, and in particular, to hyperbolic figure-eight levels.

\begin{lemma}\label{constDirection}
Assume that $(M, \omega)$ is a compact symplectic surface, possibly with boundary, and let $[\alpha] \in  \SVect^*(M)$ be a fixed point of the Euler equation such that the corresponding vorticity $F= \mathfrak{curl} [\alpha]$ is a simple Morse function. Let also $\pazocal E$ be a hyperbolic figure-eight level of $F$, and let $O \in \pazocal E$ be the singular point. Then:
\begin{longenum}
\item the fluid velocity field $u$ does not vanish in $\pazocal E \,\setminus\, \{O\}$;
\item $u$ induces the same orientation\footnote{Recall that each level of $F$ is naturally oriented as the boundary of the set of smaller values. This allows us to compare the orientations induced by $u$ on different connected components of $\pazocal E \,\setminus\, \{O\}$. } on both connected components of $\pazocal E \,\setminus\, \{O\}$ (see Figure~\ref{steadyFlow2}).
\end{longenum}

\end{lemma}

 \begin{figure}[t]
\centerline{
\begin{tikzpicture}[thick, scale = 1]
\node [vertex] (a) at (0,0) {};
\fill (a) circle [radius=1.5pt];
\begin{pgfinterruptboundingbox}
\draw [->--, -->-] (a) .. controls +(-3,2) and +(-3,-2) .. (a);
\draw  [->--, -->-]  (a) .. controls +(3,-2) and +(3,2) .. (a);
  \fill[opacity = 0.3] (0,0.3)  .. controls +(-3.4,2) and +(-3.4,-2) .. (0,-0.3);
    \fill[opacity = 0.3] (0,0.3)  .. controls +(3.4,2) and +(3.4,-2) .. (0,-0.3);
      \fill[color = white, opacity = 1] (-0.4, 0)  .. controls +(-2,1.3) and +(-2,-1.3) .. (-0.4, 0);
            \fill[color = white, opacity = 1] (0.4, 0)  .. controls +(2,1.3) and +(2,-1.3) .. (0.4, 0);
            \end{pgfinterruptboundingbox}
            \node () at (-2,0) {$\ell_1$};
                \node () at (2,0) {$\ell_2$};
                \node () at (0.1,0.3) {$O$};
               \node at (0,-1) {};
                 \node at (0,1) {};
\end{tikzpicture}
}
\caption{Neighborhood of a hyperbolic level of the vorticity function.}\label{steadyFlow2}
\end{figure}
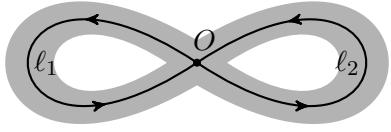
\begin{proof}
Consider a tubular neighborhood $U$ of $\pazocal E$, as depicted in Figure \ref{steadyFlow2}. Since the fluid velocity vector field $u$ is symplectic, the 1-form $\beta := i_u\omega$ is closed. Further, since $u$ is tangent to the arcs $\ell_1$ and $\ell_2$, we have
$
\beta\mid_{\ell_1} = \beta\mid_{\ell_2} = 0,
$
and, therefore,
\begin{align*}
\int_{\ell_1} \!\!\beta = \int_{\ell_2}\!\! \beta = 0\,,
\end{align*}
which implies that the 1-form $\beta$ is exact in $U$, i.e. $i_u\omega = \diff H$ for a suitable smooth function $H \colon U \to \R$ (the {stream function}).
\par
The first statement of the lemma is immediate: if one assumes that the field $u$ vanishes at some point  of 
$ \pazocal E \,\setminus \, \{O\}$, say on the arc  $ \ell_1$, then $\diff H$ vanishes at that point and hence on the whole 
arc $ \ell_1$, which contradicts   that $O$ is an isolated critical point of $H$.

\par
Now, let us prove the second statement. In a neighborhood of the point $O$, the fluid velocity field $u$ has the form

\begin{align*}
u =  \omega^{-1}\diff H = b\left(P \frac{\partial}{\partial p} -Q \frac{\partial}{\partial q} \right) + \dots \,,
\end{align*}

Similarly, we have
\begin{align*}
X_F = \omega^{-1}\diff F = \zeta'(0)\left(P \frac{\partial}{\partial p} -Q \frac{\partial}{\partial q} \right) + \dots \,,
\end{align*}
so, modulo higher order terms, we have 
$$
u \approx\frac{b}{\zeta'(0)}X_F\,.
$$
Now, note that the vector field $ X_ F$ induces the same orientation on $\ell_1$, $\ell_2$ (it is opposite to the natural orientation). Therefore, the same is true for the field $u$, q.e.d.

\end{proof}
\begin{proof}[Proof of the ``only if'' part of Theorem \ref{thm:steady}]
We need to prove that if a Morse-type coadjoint orbit admits a steady flow, then the corresponding circulation function is balanced. Assume that $[\alpha] \in  \SVect^*(M)$ is a steady solution of the Euler equation, and let $\alpha \in [\alpha]$ be a co-closed representative. Take any $3$-valent vertex $v$ of the circulation graph $\Gamma_{[\alpha]}$ corresponding to the coset $[\alpha]$. Let $e_0$ be the trunk of $v$, and let $e_1$, $e_2$ be the branches of $v$. Let also $\pazocal E$ be the figure-eight level of the vorticity function $F = \mathfrak{curl}[\alpha]$ corresponding to the vertex $v$ (see Figure~\ref{steadyFlow2}). Then for $i = 1,2$ we have
\begin{align*}
c_i(v) = \int_{\ell_i} \!\!\alpha\,,
\end{align*}
where $c_i(v)$ is defined by formula \eqref{limitCirculation}. According to Lemma \ref{constDirection}, either the direction of the fluid velocity field $u$ coincides with the natural orientation on both arcs $\ell_1$, $\ell_2$ of $\pazocal E$, or the field $u$ is opposite to the orientation of both arcs. In the case where orientations coincide, one can take the natural time parameter $t$ on integral trajectories of $u$ as a parameter on  $\ell_1$, $\ell_2$. Then, for $i=1,2$, we have
$$
c_i(v) =  \int_{\ell_i} \!\!\alpha =  \int\limits_{-\infty}^{+\infty} \alpha(\dot \ell_i(t)) \,\diff t =  \int\limits_{-\infty}^{+\infty} \alpha(u(t))\, \diff t = \int\limits_{-\infty}^{+\infty} (u(t),u(t))\,\diff t\,,
$$
where $(\,,)$ is the dot product given by the metric on $M$. So, in this case the numbers $c_i(v)$ are both positive. Similarly, in the case of opposite orientation of $u$ and the arcs, $c_i(v)$ are both negative.
Thus, in both cases the numbers $c_i(v)$ are  non-zero and have the same sign, as desired.
\end{proof}
\par
\medskip

\subsection{Construction of steady flows}
In this section we prove that if a circulation function corresponding to a Morse-type coadjoint orbit $\pazocal O \subset \SVect^*(M)$ is balanced, then the orbit $\pazocal O$ admits a steady flow, i.e. there exists a metric $g$ on $M$ compatible with the symplectic structure $\omega$ and such that the corresponding Euler equation has a steady solution on $\pazocal O$. We divide the proof into several statements.
The first lemma is well-known being a particular case of the Arnold-Liouville theorem:

\begin{lemma}\label{ALT}
Assume that $(C, \omega)$ is a compact symplectic cylinder, and let $F \colon C \to \R$ be a smooth function constant 
on the boundary of $C$. Assume also that $\diff F \neq 0$ on $C$.
Then there exist functions
$$
S \colon C \to \R, \quad \Theta \colon C \to S^1 = \R \, / \, 2\pi\Z,
$$
called the \textit{action-angle variables}, such that $\omega = \diff S\wedge \diff \Theta$, and $F = \zeta(S)$ where $\zeta$ is a smooth function in one variable such that $\zeta'(S) \neq 0$.
\end{lemma}

Recall that to construct a steady flow with a prescribed vorticity $F$ on a symplectic surface $(M, \omega)$ 
it suffices to define an $F$-steady triple $(\alpha, J, H)$ on $M$, see Definition \ref{defn:steady-triple}.
Here  $\alpha$ is a $1$-form, $J$ is an almost complex structure compatible with $\omega$, and $H = H(F)$ satisfying
$J^*\alpha = -\diff  H$ and $\diff \alpha = F  \omega$.

\begin{proposition}\label{step0}
Assume that $(C, \omega)$ is a symplectic cylinder, and let $F \colon C \to \R$ be a smooth function constant on the boundary of $C$. Assume also that $\diff F \neq 0$ on $C$. Then there exists an $F$-steady triple $(\alpha, J, H)$ on $C$. Moreover, for any $F$-level $\ell \subset C$ and any constant $c$, the form $\alpha$ can be chosen in such a way that
$
\int_{\ell} \alpha = c\,.
$

\end{proposition}

\begin{proof}
We work in action-angle coordinates $S, \Theta$ from Lemma \ref{ALT}. Without loss of generality, we may assume that $S = 0$ on $\ell$. By setting 
\begin{align}\label{stTripleCyl}
\alpha := \eta(S) \diff \Theta\,, \quad H := -\int \eta(S) \diff S\,,\quad\text{ and }\,\,
 J:= \left(\begin{array}{cc}0 & -1 \\1 & 0\end{array}\right)
\end{align}
we provide $J^*\alpha = -\diff H$ for any smooth function $\eta(S)$. Now by specifying this function to be
$$
\eta(S) := \frac{c}{2\pi} + \int\limits_0^S \zeta(\xi)\diff \xi
$$
we also obtain  
$  \diff \alpha = F  \omega $ for $F = \zeta(S)$, as well as $\int_{\ell} \alpha = c$,  
as desired. 
\end{proof}

\begin{proposition}\label{step1}
Let $(M, \omega)$ be a symplectic surface, and let $F \colon M \to \R$ be a smooth function on $M$. Assume that $O \in M\, \setminus \, \partial M$ is a non-degenerate minimum or maximum singular point of $F$. Then there exists an $F$-steady triple $(\alpha, J, H)$ in the neighborhood of $O$.
 \end{proposition}
 
\begin{proof} This proposition can be thought of as a modification of the one above for $c=0$, but requires 
an independent proof of smoothness at $O$. 
By Lemma \ref{MDL}, there exist local coordinates $P,Q$ in the neighborhood of $O$ such that 
$\omega = \diff P \wedge \diff Q$ and $F = \zeta(S)$ where $S=\frac{1}{2}(P^2 + Q^2)$. 
We use the polar coordinates $(\sqrt {2S}, \Theta)$  related to $(P,Q)$ as action-angle variables, 
 $\omega = \diff P \wedge \diff Q=\diff S \wedge \diff \Theta$. 
Then we set $$\eta(S) := \int_0^S \zeta(\xi) \diff \xi\,,$$
and define a steady triple $(\alpha, J, H)$ using formulas \eqref{stTripleCyl}. %Then $J^*\alpha = -\diff H$ and 
%$\diff \alpha = \eta'(S)\diff S  \wedge\diff \Theta  =   F \diff S  \wedge\diff \Theta = F\omega$,
%as desired. 
\end{proof}

 \medskip
 
Before we turn to constructing $F$-steady triples near hyperbolic levels, let us point out the following 
key property of the sign of $ \diff H/\diff F $.  %, where ``$\mathrm{sign}$'' stands for the signum function.

\begin{lemma}\label{lem:sign}
For a steady flow $\alpha$ with a stream function $H=H(F)$  expressed in terms of vorticity $F$, the sign of the derivative 
of the stream function with respect to vorticity is minus the sign of the circulation function. More precisely,  
for any nonsingular $F$-level $\ell\subset M$
$$
 \sgn \diffFXp{H}{F} = -\sgn \int_{\ell} \alpha\,.
$$
\end{lemma}

Equivalently, in terms of the Reeb graph $(\Gamma,f)$ of the function $F$ on $M$ 
and the circulation function $\circulation(x)$ on the graph, defined 
by integrals of the 1-form $\alpha$ over levels $\ell_x= \pi^{-1}(x)$ for interior points $x \in \Gamma$, one has
\begin{align}\label{streamVorticityCirculation}
\sgn \diffFXp{H}{F}(f(x)) = -\sgn \circulation(x)\,.
\end{align}
The proof in terms of $F$-steady triples directly follows from the relation  $J^*\alpha = -\diff H$. 

\smallskip

Now we study the vicinity of hyperbolic levels of vorticity functions.

\begin{proposition}\label{step2}
Let $(M, \omega)$ be a symplectic surface, possibly with boundary, and let $F \colon M \to \R$ be a smooth function on $M$. Assume that $O \in M\, \setminus \, \partial M$ is a non-degenerate hyperbolic singular point of~$F$. Then there exists an $F$-steady triple $(\alpha, J, H)$ in the neighborhood of $O$. Moreover, for any $\eps = \pm 1$, the triple $(\alpha, J, H)$ can be chosen in such a way that $\sgn (\diff H/\diff F) = \eps$.
%$$
%\sgn \diffFXp{H}{F} = \eps\,,
%$$
 \end{proposition}
 
 \begin{proof}
By Lemma \ref{MDL}, there exist local coordinates $P, Q$ in the neighborhood of $O$ such that $\omega = \diff P \wedge \diff Q$, and $F = \zeta(S)$ where $S=PQ$, and $\zeta'(0) \neq 0$. Without loss of generality, we may assume that $\zeta'(0) > 0$ (if not, we replace the chart $P, Q$ with $Q, -P$). Let
$$
\alpha := (\eta(S)Q +  cP)\diff P - (  \eta(S)P + cQ )\diff Q\,\quad \text{ for }\,\, 
\eta(S) := -\frac{1}{2S}\int\limits_0^S \zeta(\xi) \diff \xi\,,
$$
where  $c$ is any constant such that $\sgn c = \eps$ and $|c| > |\eta(S)|$ in a sufficiently small 
neighborhood of the point $O$. Then we have 
$$
\diff \alpha = - 2(S\eta'(S) + \eta(S)) \,\diff P \wedge \diff Q = F \omega\,,
$$
as desired. Further, for $\eps = 1$ we set
$$
J :=  \frac{1}{\sqrt{c^2 - \eta(S)^2}} \left(\begin{array}{cc}\eta(S) & -c \\c & -\eta(S)\end{array}\right)\,, \quad H(S) := \int \sqrt{c^2 - \eta(S)^2}\, \diff S\,.
$$
Then $J$ is a complex structure compatible with $\omega$, and $J^*\alpha = -\diff H$. Furthermore, we obtain
$$
\sgn \frac{\diff H}{\diff F} = \sgn\left( \frac{\diff H}{\diff S} :  \frac{\diff F}{\diff S} \right)= \sgn \frac{\sqrt{c^2 - \eta(S)^2}}{\zeta'(S)} = 1\,,
$$
as required. Similarly, for $\eps = - 1$ we simply need to change the signs  of $J$ and $H$.
\end{proof}

\begin{proposition}\label{step3}
Let $(M, \omega)$ be a compact symplectic surface, possibly with boundary, and let $F \colon M \to \R$  be a simple Morse function on $M$. Let also $\pazocal E$ be a figure-eight level of $F$. Then there exists a steady $F$-triple $(\alpha, J, H)$ in the neighborhood of $\pazocal E$. Moreover, for any non-zero numbers $c_1, c_2 \in \R$ having the same sign, the triple $(\alpha, J, H)$ can be chosen in such a way that
$$
\int_{\ell_i} \alpha = c_i
$$
where $\ell_i$ are the loops constituting the level $\pazocal E$ (see Figure \ref{steadyFlow2}).
\end{proposition}

The proof of this proposition is a rather technical explicit construction, and we give it in Appendix A.% \ref{app:proof-eight}.

\begin{proposition}\label{step4}
Let $(C,\omega)$ be a symplectic cylinder with boundary $\partial C = \ell_2 - \ell_1$, and let  $F \colon C \to \R$ be a smooth function without critical points on $C$ and constant on the boundary components:
$
F\vert_{\ell_1} = c_1 < c_2 =  F\vert_{\ell_2} 
$.
%and such that $\diff F \neq 0$ on $C$.
Let also $\alpha_1, \alpha_2$ be $1$-forms defined in  neighborhoods of the cycles $\ell_1, \ell_2$ respectively such that
\begin{align}\label{diffOfAlpha12}
\diff \alpha_1 = \diff \alpha_2 = F\omega\,.
\end{align}
Assume also that  in those neighborhoods the expression $\omega^{-1}(\alpha_i, \diff F) $ does not vanish, 
i.e. 1-forms $\alpha_i $ and $ \diff F$ are  linearly independent at every point of the neighborhoods.
Furthermore, consider the function
$$
\eta(z) :=  \int_{\ell_1} \alpha_1 + \int_{C_z} F\omega\,,
$$
where $C_z$ is the set of points where $F \leq z$. We  assume that 
\begin{align}\label{givenStokes}
\eta(c_2) =  \int_{\ell_2} \alpha_2 \,,
\end{align}
and that $ \eta(z) $ does not change sign in the interval $[c_1, c_2]$.

Then there exists a $1$-form $\alpha$ on $C$ which ``interpolates" between $\alpha_i $, i.e. it 
is linearly independent with $\diff F$, coincides with $\alpha_i$ in a sufficiently small neighborhood of $\ell_i$, and satisfies
$
\diff \alpha = F\omega\,.
$
\end{proposition}

\begin{proof}
We work in action-angle coordinates $S, \Theta$ from Lemma \ref{ALT}. 
In the neighborhood of $\ell_i$ we have 
$$
\alpha_i = A_i(S, \Theta)\diff S + B_i(S, \Theta) \diff \Theta\,.
$$
By taking the direct image of the second component under the projection to the $S$-axis and recalling that $F= \zeta(S)$
we see that 
$$
\int\limits_{0}^{2\pi} B_i(S, \Theta) \diff \Theta = \eta (\zeta(S))
$$
(for any $S$ where the left-hand side is well-defined), as follows  from~\eqref{diffOfAlpha12},~\eqref{givenStokes}, 
and the Stokes formula. The condition of linear independence implies that the functions $B_1, B_2$ do not vanish. 
Furthermore, they have the same sign as  the function $\eta$, since 
$$
\int\limits_0^{2\pi} B_i(S(\ell_i), \Theta) \diff \Theta = \int_{\ell_i} \alpha_i = \eta(c_i)\,.
$$
Consider the convex set $\pazocal B$   of smooth functions $B(S, \Theta)$ on the cylinder $C$ which have the same sign as the function $\eta$
 and the same pushforward $ \eta (\zeta(S))$ for the 1-forms $B(S, \Theta)\diff \Theta $. The set $\pazocal B$ is nonempty as it contains a 
function of $S$ only:  $B_0(S, \Theta) := \frac{1}{2\pi}\eta(\zeta(S))$. Therefore, using a partition of unity, one can construct a function  $\tilde B(S, \Theta) \in \pazocal B$
which coincides with $B_i$ in a neighborhood of $\ell_i$ for $i= 1,2$. Then one can restore the desired 1-form
$$
\alpha := A(S, \Theta)\diff S + \tilde B(S, \Theta) \diff \Theta\,,
$$
satisfying $ \diff \alpha = F\omega$ 
by Poincar\'e's lemma with parameters: the 1-form $(\frac{\partial}{\partial S}\tilde B(S, \Theta)- \zeta(S)) \diff \Theta$  is closed on levels sets of $F$ and has zero periods, and hence exact, i.e. for any fixed $S$ represented as $\diff_\Theta A(S, \Theta)$ for an appropriate function 
$A(S, \Theta)$ on the cylinder.
\end{proof}

\begin{definition}
Let $(\Gamma, f, \mu)$ be a measured Reeb graph, and let $\circulation$ be a balanced circulation function on $\Gamma$. 
Denote by $\wave V \subset \Gamma$ the set of {\it exceptional points} of the graph, defined as the union of all vertices and 
all points $x$ of the graph where either $f(x)=0$ or $\circulation(x)=0$ (or both):
$$
\wave V := V(\Gamma) \cup  \{ x \in \Gamma \, \mid f(x) = 0\} \cup \{ x \in \Gamma\, \mid \circulation(x) = 0\} .
$$ 
\end{definition}

\begin{definition}
By a \textit{$1$-form} on a Reeb graph $(\Gamma, f)$  we mean an expression which can locally be written as $g(f) \diff f$, where $g(f)$ is a smooth function of $f$. A $1$-form $\beta$ is \textit{exact} if for any cycle $\ell$ in $\Gamma$ we have
$
\int_\ell \beta = 0\,.
$
\end{definition}
\begin{proposition}\label{step5}
Let $(\Gamma, f, \mu)$ be a measured Reeb graph,  $\circulation$  a balanced circulation function on $\Gamma$, and 
$\wave V $ the set of the corresponding exceptional points of the graph. 
Assume that in the neighborhood $U_i$ of each exceptional point $x_i \in  \wave V $ there is  a $1$-form $\beta_i$ such that for each $x \in U_i \, \setminus \, \{x_i \}$, we have
\begin{align}\label{compatibilityWithCirculation}
\sgn \frac{\beta_i}{\diff f} = -\sgn \circulation\,.
\end{align}
Then there exists a $1$-form $\beta$ on $\Gamma$ which does not vanish in $\Gamma \, \setminus \,  \wave V$, coincides with $\beta_i$ in a neighborhood of $x_i \in \wave V$, and such that the form $f\beta$ is exact.
\end{proposition}

The proof of Proposition \ref{step5} is based on the following lemma:

\begin{lemma}\label{nonCycleProp}
Assume that $\Gamma$ is an oriented graph, and let $\eps \colon E(\Gamma) \to \{ \pm 1\}$ be a function such that for any subset $E' \subset E(\Gamma)$, we have
\begin{align}
\label{nonCycle}
\partial \sum_{e \,\in \, E'} \eps(e)e \neq 0\,.
\end{align}
Then there exists a $1$-coboundary $\xi \colon E(\Gamma) \to \R$ such that $\sgn \xi = \eps$.
\end{lemma}

\begin{proof}
Reverse the orientation of all edges $e$ such that $\eps(e) = -1$. Then from condition \eqref{nonCycle} it follows that the obtained graph has no directed cycles. Therefore, there exists a $0$-cochain $\eta$ such that $\eta(w) > \eta(v)$ if there is an edge going from $v$ to $w$, and it is easy to see that the $1$-coboundary $\xi := \delta \eta$, where $\delta$ is the coboundary operator, has the desired property.
\end{proof}

\begin{proof}[Proof of Proposition \ref{step5}]
 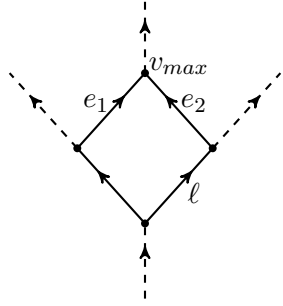
\begin{figure}[t]
\centerline{
\begin{tikzpicture}[thick, scale = 1]
    \node [vertex] at (0,0) (nodeC) {};
        \node [] at (0.45,0.1) (nodeV) {$v_{max}$};
                    \node [] at (-0.65,-0.4) () {$e_1$};
                \node [] at (0.65,-0.4) () {$e_2$};
                  \node [] at (0.65,-1.6) () {$\ell$};
        \node [vertex] at (-0.9,-1) (nodeD) {};
          \draw  [->-, dashed] (nodeD) -- (-1.8, 0);
              \node [vertex] at (0.9,-1) (nodeE) {};
                    \draw  [->-, dashed] (nodeE) -- (1.8, 0);
     \node [vertex] at (0,1) (nodeF) {};          
          \node [vertex] at (0,-2) (nodeG) {}; 
                \draw  [->-, dashed] (0, -3) -- (nodeG);     
    \draw  [->-] (nodeD) -- (nodeC);
        \draw  [->-] (nodeE) -- (nodeC);
           \draw  [->-, dashed] (nodeC) -- (nodeF);
    \fill (nodeC) circle [radius=1.5pt];
        \fill (nodeD) circle [radius=1.5pt];
            \fill (nodeE) circle [radius=1.5pt];
                \fill (nodeG) circle [radius=1.5pt];
                \draw  [->-] (nodeG) -- (nodeD);
                      \draw  [->-] (nodeG) -- (nodeE);
\end{tikzpicture}
}
\caption{A maximum of $f$ on a cycle $\ell$ in a measured Reeb graph.}\label{maxf}
\end{figure}
Consider the graph $\wave \Gamma$ obtained from $\Gamma$ by regarding all points $x \in \Gamma$ such that $f(x) = 0$ or $\circulation(x) = 0$ as $2$-valent vertices. Then the product $f\circulation$ has the same sign at interior points of any edge $e \in \wave \Gamma$. Define a function $\eps \colon E(\wave \Gamma) \to \{ \pm 1\}$ by setting  \begin{align}\label{epsDef}\eps(e) := -\sgn f\circulation\vert_e\,. \end{align}
Thanks to condition \eqref{compatibilityWithCirculation}, the forms $\beta_i$ defined near vertices of the graph $\wave \Gamma$ extend to a global form $\beta$ such that
$$
\sgn \frac{\beta}{\diff f}(x) = -\sgn \circulation(x)
$$
for any interior point $x \in \wave \Gamma$. For any edge $e \in E(\wave \Gamma)$, we have
$$
\sgn \int_e f\beta = \eps(e)\,,
$$
where $\eps(e)$ is given by formula \eqref{epsDef}.  
Moreover, it is easy to see that for any $1$-cochain $ \xi \colon E(\wave \Gamma) \to \R$ such that $\sgn \xi(e) = \eps(e)$, one can choose the form $\beta$ in such a way that
\begin{align}\label{integralOverEdge}
 \int_e f\beta = \xi(e)\,.
\end{align}
We claim that the function $\eps(e)$ satisfies condition \eqref{nonCycle}. Indeed, assume that
$$
\ell = \sum_{e \,\in \, E'} \eps(e)e
$$
is a cycle. 
Consider the point $v_{max} \in \ell$ where $f$ is maximal. Then $v_{max}$ has to be a $3$-valent vertex with two incoming edges $e_1, e_2 \in E(\wave \Gamma) $ (see Figure \ref{maxf}). Note that since the circulation function $\circulation$ is balanced, the product $ f\circulation$ has the same sign at $e_1$ and $e_2$. Therefore, we have $\eps(e_1) = \eps(e_2)$, which contradicts $\ell$ being a cycle.\par
Thus,  since $\eps(e)$ satisfies \eqref{nonCycle}, by Lemma \ref{nonCycleProp} the $1$-cochain $\xi$ entering \eqref{integralOverEdge} can be chosen to be a coboundary. The latter implies that $\beta$ can be chosen in such a way that $f\beta$ is exact, q.e.d.
\end{proof}

\begin{proof}[Proof of the ``if'' part of Theorem \ref{thm:steady}]
Let $(M, \omega)$ be a compact connected symplectic surface, possibly with boundary. Let also $\pazocal O \subset \SVect^*(M)$ be a Morse-type coadjoint orbit, and let $(\Gamma, f, \mu, \circulation)$ be the corresponding circulation graph. Assume that the circulation function $\circulation$ is balanced. We need to show that there exists a metric $g$ on $M$ which is compatible with $\omega$ and such that the corresponding Euler equation has a fixed point on $\pazocal O$.\par
Take any $[\wave \alpha] \in \pazocal O$, and let $F = \mathfrak{curl} [\wave \alpha]$. Then the measured Reeb graph $\Gamma_F$ associated with $F$ coincides with $(\Gamma, f, \mu)$, and we have a projection $\pi \colon M \to \Gamma$. As in Proposition \ref{step5}, let $\wave V $ be the set of exceptional points, 
$$
\wave V := V(\Gamma) \cup  \{ x \in \Gamma \, \mid f(x) = 0\} \cup \{ x \in \Gamma\, \mid \circulation(x) = 0\}\,,
$$ 
and let $U_i$ be a neighborhood  of $x_i \in \wave V$. Now, we use Propositions \ref{step0}, \ref{step1}, \ref{step3} to construct an $F$-steady triple $(\alpha_i, J_i, H_i)$ in $\pi^{-1}(U_i)$ such that:
\begin{longenum}
\item If $x_i \in \wave V$ is a boundary vertex or an interior point of $\Gamma$, then
$$
\int_{\ell_i} \alpha_i = \circulation(x_i)\,,
$$
where $\ell_i = \pi^{-1}(x_i)$.
\item If $x_i \in \wave V$ is a $3$-valent vertex of $\Gamma$, and $e_1$, $e_2$ are branches of $x_i$, then for $j = 1,2$ we have
$$
\int_{\ell_{ij}} \alpha_i = c_j(x_i)
$$
where $\ell_{ij} = \pi^{-1}(x_i) \cap \partial \pi^{-1}(e_j)$, and $c_j(x_i)$ is the limit of the circulation function $\circulation(x)$ as $x$ tends to $x_i$ along the edge $e_j$.
\end{longenum}
It follows from our construction that for any $x \in U_i \, \setminus \, x_i$ we have
$$
\int_{\ell_x} \alpha_i = \circulation(x)\,,
$$
where $\ell_x = \pi^{-1}(x)$. Furthermore, from the relation $J^*\alpha_i = -\diff H_i(F)$, it follows that if $\alpha_i$ is linearly dependent with $\diff F$ at some point $Z$ lying on a non-singular level $\pi^{-1}(x)$, then $\alpha \equiv 0$ on $\pi^{-1}(x)$, and thus $\circulation(x) = 0$. Therefore, $\alpha_i$ and $\diff F$ are independent at all points of $M \, \setminus \,  \pi^{-1}(\wave V)$ (the latter also follows from a more general formula \eqref{streamVorticityCirculation} above).

\par
Moreover, using Proposition \ref{step4}, we extend the forms $\alpha_i$ to a form $\alpha$ defined on the whole $M$ which is linearly independent with $\diff F$ at all points of $M \, \setminus \,  \pi^{-1}(\wave V)$ and satisfies $\diff \alpha = F\omega$. Note that for any interior point $x \in \Gamma$, we have
$$
\int_{\ell_x} \alpha = \circulation(x)\,,
$$
where $\ell_x = \pi^{-1}(x)$. By Lemma \ref{lem:sign} %(see formula \eqref{streamVorticityCirculation})
one obtains
%\begin{align}\label{streamVorticityCirculation2}
%\sgn \diffFX{H_i}{F}(f(x)) = -\sgn \circulation(x)
%\end{align}
$\sgn {\partial H_i}/{\partial F}(f(x)) = -\sgn \circulation(x)$
for any $x \in U_i \, \setminus \, \{x_i\}$. Therefore, descending the forms $\diff H_i$ from $\pi^{-1}(U_i)$ to $U_i$, one obtains a family of $1$-forms $\beta_i$ satisfying the condition of Proposition \ref{step5}. So, there exists a $1$-form $\beta$ on $\Gamma$ which does not vanish in $\Gamma \, \setminus \,  \wave V$, coincides with $\beta_i$ in a neighborhood of $x_i \in \wave V$, and such that the form $f\beta$ is exact. Lifting the form $\beta$ back to $M$, we obtain a form $ \gamma$ which does not vanish in $M \, \setminus \,  \pi^{-1}(\wave V)$, coincides with $\diff H_i$ in the neighborhood of the fiber $\pi^{-1}(x_i)$, and such that the form $F \gamma$ is exact. Now, we extend almost complex structures $J_i$ to a global almost complex structure $J$ by setting
$J^* \alpha := -\gamma, \, J^*\gamma := \alpha.$
Further, define a metric $g$ on $M$ be setting $g(u_1,u_2) := \omega(u_1, Ju_2)$. Then, from the formula $\diff(J^*\alpha) = -\diff \gamma = 0$ it follows that  $\alpha$ is co-closed. For the fluid velocity field $u = \alpha^\sharp$, we have
$$
L_u [\alpha]  = [i_u \diff \alpha] = [F i_u \omega] = -[FJ^*\alpha] = [F\gamma] = 0\,
$$
so $[\alpha]$ is a steady solution of the Euler equation lying on the orbit $\pazocal O$, as desired (we have $[\alpha] \in \pazocal O $ since $\diff \alpha = F\omega$, and the circulation function corresponding to $[\alpha] $ coincides, by construction, with the prescribed function $\circulation$).

\end{proof}

%%%%%%%%%%%%%%%%%%%%%%%%%%%%
\medskip

\subsection{Steady flows on closed surfaces}\label{sect:closed}
\begin{definition}
Let $\Gamma$ be a measured Reeb graph with no boundary vertices, and let $\circulation$ be a circulation function on 
$\Gamma$. We say that $\circulation$ is \textit{totally negative} if for any $3$-valent vertex $v$ of $\Gamma$, 
the limits $c_0(v)$, $c_1(v)$, $c_2(v)$ of $\circulation$ at $v$ are negative.
\end{definition}

\begin{remark}
Note that total negativity in particular implies that $\circulation(x) < 0$ at any interior point $x \in \Gamma \, \setminus\, V(\Gamma)$. Indeed, for any edge $e \in \Gamma$ the total negativity condition and the Kirchhoff condition \eqref{3valentcirc} at $1$-valent vertices imply that $\circulation(x)$ has non-positive limits at endpoints of $e$. At the same time, from property \eqref{stokes} of circulation functions and the monotonicity of $f$ along $e$ it follows that the function $\circulation$ restricted to $e$ is concave down. Therefore, we indeed have $\circulation< 0$ at all points of $e$.
\end{remark}
 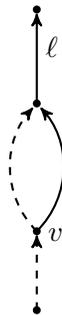
\begin{figure}[b]
\centerline{
\begin{tikzpicture}[thick, scale = 1]
    \node [vertex] at (7,0) (nodeC) {};
    \node [vertex]  at (7,1.05) (nodeD) {};
    \node [vertex] at (7,2.75) (nodeE) {};
    \node [vertex]  at (7,4) (nodeF) {};
    \draw  [a, dashed] (nodeC) -- (nodeD);
    \fill (nodeC) circle [radius=1.5pt];
    \fill (nodeD) circle [radius=1.5pt];
    \fill (nodeE) circle [radius=1.5pt];
    \fill (nodeF) circle [radius=1.5pt];
    \draw  [a] (nodeD) arc (-45:44:1.2cm);
    \draw  [a, dashed] (nodeD) arc (225:136:1.2cm);
    \draw  [a] (nodeE) -- (nodeF);
    \node at (7.6,2) () {$$};
     \node at (7.2,3.5) () {$\ell$};
     \node at (7.25, 1) () {$v$};
\end{tikzpicture}
}
\caption{A maximal directed path $\ell$ starting at a $3$-valent vertex $v$.}\label{path2}
\end{figure}
It turns out that for closed surfaces Theorem \ref{thm:steady} can be reformulated in the following way:

\begin{theorem}\label{thm:steadyclosed}
Let $(M, \omega)$ be a closed connected symplectic surface, and let  $\pazocal O \subset \SVect^*(M)$ be a Morse-type coadjoint orbit. Then $\pazocal O$ admits a steady flow if and only if the corresponding circulation function is totally negative.
\end{theorem}

\begin{proof}
It suffices to show that if $\Gamma$ is a measured Reeb graph with no boundary vertices, then a circulation function 
$\circulation$ on $\Gamma$ is balanced if and only if it is totally negative. Clearly, if  $\circulation$ is totally negative, 
then it is balanced. Let us prove the converse statement. Let $\circulation$ be a balanced circulation function on $\Gamma$. 
Assume that there exists a $3$-valent vertex $v \in \Gamma$ such that $c_i(v) > 0$ for $i = 0,1,2$. Let us first assume 
that $f(v) \geq 0$. Consider any maximal directed path $\ell$ in $\Gamma$ starting at the vertex $v$ (see Figure \ref{path2}). 
Then $f \geq 0$ at points of the path $\ell$, and the circulation function $\circulation$ has non-negative limits at endpoints 
of $\ell$, since it ``integrates" $f$ along $\ell$. However, at the top endpoint the circulation function tends to $0$.
So, the same argument as in Example \ref{GKEx} shows that $\circulation$ must change 
sign at some $3$-valent vertex $w \in \ell$, which contradicts the fact that  $\circulation$ is balanced. 
Similarly, if $f(v) \leq 0$, then one considers a maximal directed path ending at $v$ and applies the same argument 
to obtain a contradiction. So, we indeed have $c_i(v) < 0$ for all $3$-valent vertices $v \in \Gamma$, q.e.d.
\end{proof}

\begin{example}\label{example:circtorus2}
Consider the  circulation graph $(\Gamma, f, \mu, \circulation)$ for a Morse-type coadjoint orbit 
$\pazocal O \subset \SVect^*(T^2)$ on a torus depicted in Figure \ref{circTorus}, cf. Example \ref{example:circtorus}.
Let $$a_i := \int_{e_i} f\diff \mu\,.$$
The space of circulation functions on $\Gamma$ is one-dimensional and the numbers $a_i$ satisfy the relation
$a_1 + a_2 + a_3 + a_4 = 0$ by Proposition~\ref{circFuncs}.
\par
 
The set of totally negative circulation functions on $\Gamma$ is described by linear inequalities
$$
z < 0, \quad a_1 - z< 0, \quad a_2 + z < 0, \quad a_1 + a_3 -z < 0\,,
$$
i.e. the circulation function $\circulation$ is totally negative if and only if $z \in I_1 \cap I_2$, where $I_1$, $I_2$ are open intervals $I_1 = (a_1, 0)$, $I_2 = (a_1 + a_3, -a_2)$. Using that $f$ is monotonous along edges of $\Gamma$ and the zero average condition $a_1 + a_2 + a_3 + a_4 = 0$, one can easily prove that $a_1 < 0$, and $a_1 + a_3 < -a_2$, so that the intervals  $I_1$, $I_2$ are non-empty. However, these intervals do not need to intersect, so depending on $a_i$'s, the set of totally negative circulation functions on $\Gamma$ may be empty. Furthermore, even if intervals $I_1$, $I_2$ do intersect, the set of totally negative circulation functions on $\Gamma$ is only a bounded interval, while the set of all circulation functions is a line $\R^1$. The latter implies that, roughly speaking, most of Morse-type coadjoint orbits of $\SDiff(T^2)$ corresponding to the graph in Figure \ref{circTorus} do not admit steady solutions of the Euler equation.
\end{example}

The following theorem generalizes the conclusion of Example \ref{example:circtorus2}: 

\begin{theorem}\label{thm:polytope}
Assume that $(\Gamma, f, \mu)$ is a measured Reeb graph with no boundary vertices. Then the set of totally negative circulation functions on $\Gamma$ is a (possibly empty) open convex polytope in the affine space of all circulation functions. In particular, the set of totally negative circulation functions on $\Gamma$ is bounded.
\end{theorem}

\begin{proof}
It is obvious from the definition that the set of totally negative circulation functions on $\Gamma$ is an open convex polyhedron. So, it suffices to prove that the latter set is bounded. Since a circulation function $\circulation$ is uniquely determined by numbers $\circulation^+(e)$ (recall that the notation $\circulation^+(e)$ stands for the limit of the circulation function $\circulation(x)$ as $x$ tends to the endpoint of the edge $e$ along $e$), it suffices to show that for any edge $e \in \Gamma$ there exist numbers $m, M \in \R$ such that for any totally negative  circulation function $\circulation$ on $\Gamma$ we have $m \leq \circulation^+(e) \leq M$.\par
Consider the following poset structure on the set of edges of $\Gamma$. Say that $e_1 \leq e_2$ if there exists a directed path $\ell$ in $\Gamma$ whose first edge is $e_1$, and whose last edge is $e_2$. Let us also define the height $h(e)$ of an edge $e \in \Gamma$ as the maximal length (i.e. the number of edges) of a directed path whose last edge is $e$. Using properties of circulation functions, the total negativity condition, and induction by height, it can be easily shown that for any edge  $e \in \Gamma$ and any totally negative circulation function $\circulation$ on $\Gamma$ one has
$$
 \circulation^+(e) \geq \sum_{e' \leq e} \int_e f\diff \mu\,.
$$
At the same time, by total negativity condition one has $ \circulation^+(e) \leq 0$, so the set of totally negative circulation functions on $\Gamma$ is indeed bounded, q.e.d.
\end{proof}

\begin{corollary}\label{cor:bound}
Assume that $M$ is a closed connected symplectic surface of genus $\geq 1$. Let also $(\Gamma, f, \mu)$ be a measured Reeb graph compatible with $M$ and satisfying the condition 
$
\int_\Gamma f\diff \mu = 0
$. 
Then the set of coadjoint orbits of $\SDiff(M)$ corresponding to the graph $\Gamma$ and admitting steady solutions of the Euler equation is an open convex polytope in the affine space of all coadjoint orbits of $\SDiff(M)$ corresponding to $\Gamma$.
\end{corollary}

\begin{remark}
Loosely speaking, Corollary \ref{cor:bound} states that if $M$ is of genus $\geq 1$, then ``most'' Morse-type coadjoint orbits of $\SDiff(M)$ do not admit steady solutions of the Euler equation. Note that the conclusion of Corollary \ref{cor:bound} remains true for the sphere as well, however in this case the affine space of all coadjoint orbits of $\SDiff(M)$ corresponding to $\Gamma$ is a point (i.e., there is a one-to-one correspondence between coadjoint orbits and measured Reeb graphs), so the statement of Corollary \ref{cor:bound} becomes trivial.
\end{remark}

  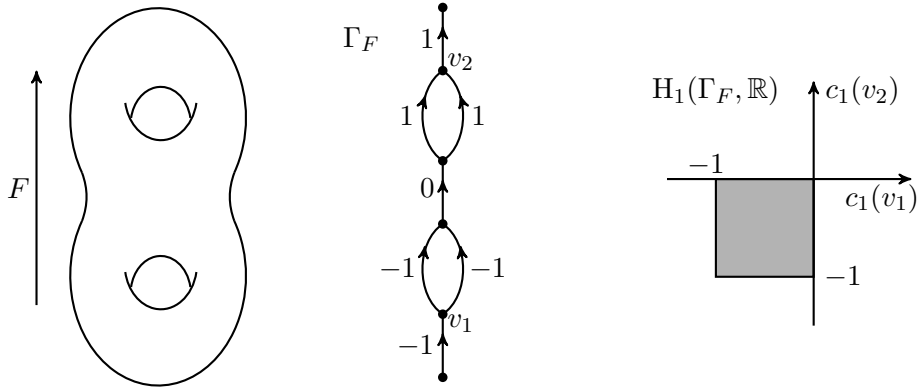
\begin{figure}[t]
\centerline{
\begin{tikzpicture}[thick, scale=1.2, rotate = -90]
\node at (3.2,1) ()
{
\begin{tikzpicture}[thick, xscale=0.8, yscale = 1, rotate = 90]
\draw   (1.68,3.2) arc (60:300:1.46cm);
\draw   (2.32,3.2) arc (120:-120:1.46cm);
\draw (1.68,3.2)  .. controls (1.9, 3.1) and (2.1, 3.1) ..  (2.32,3.2) ;
\draw (1.68,0.67)  .. controls (1.9, 0.77) and (2.1, 0.77) ..  (2.32,0.67) ;
    \draw   (1,1.3) arc (260:100:0.6cm);
    \draw   (0.8,1.4) arc (-80:80:0.5cm);
        \draw   (3.25,1.3) arc (260:100:0.6cm);
    \draw   (3.05,1.4) arc (-80:80:0.5cm);
%\draw   (1.59,2) arc (45:135:0.9cm);
%\draw   (1.7,2.1) arc (-35:-145:0.9cm);
%\draw   (2.41,2) arc (135:45:0.9cm);
%\draw   (2.3,2.1) arc (-145:-35:0.9cm);
%\draw    [densely dotted] (2.75, 1.95) arc (65:115:1.85cm);
%\draw   (2.75, 1.95) arc (-65:-115:1.85cm);
%\draw   [densely dotted] (4.5, 1.95) arc (60:120:1cm);
%\draw (4.5, 1.95) arc (-60:-120:1cm);
%\draw   [densely dotted]  (0.5, 1.95) arc (60:120:1cm);
%\draw  (0.5, 1.95) arc (-60:-120:1cm);
%\draw   (1.68,3.2) arc (60:300:1.46cm);
%\draw   (2.32,3.2) arc (120:-120:1.46cm);
%\draw (1.68,3.2)  .. controls (1.9, 3.1) and (2.1, 3.1) ..  (2.32,3.2) ;
%\draw (1.68,0.67)  .. controls (1.9, 0.77) and (2.1, 0.77) ..  (2.32,0.67) ;
%\draw   (1.59,2) arc (45:135:0.9cm);
%\draw   (1.59,2) arc (-45:-135:0.9cm);
%\draw   (2.41,2) arc (135:45:0.9cm);
%\draw   (2.41,2) arc (-135:-45:0.9cm);
%\draw  (2.4, 2) arc (50:130:0.63cm);
%\draw  [ densely dashed] (2.4, 2) arc (-50:-130:0.63cm);
%\draw (4.5, 2) arc (50:130:0.63cm);
%\draw  [densely dashed](4.5, 2) arc (-50:-130:0.63cm);
%\draw   (0.3, 2) arc (50:130:0.63cm);
%%\draw   [densely dashed] (0.3, 2) arc (-50:-130:0.63cm);
%\draw[->](2.3,2)  .. controls (2.3, 4.5) and (2.1, 4.5) ..  (2.1,3.2);
%\draw(2.3,2)  .. controls (2.3, -0.5) and (2.1, -0.5) ..  (2.1,0.7);
%\node at (2.75, 3.9) (nodeF) {$180^{\circ}$};
%\node at (0.1, 1.5) (nodeG) {$C_1$};
%\node at (2, 1.5) (nodeH) {$C_2$};
%\node at (4, 1.5) (nodeI) {$C_3$};
%\node at (0.8, 3.7) (nodeK) {$P_1$};
%\node at (0.8, 0.1) (nodeL) {$P_2$};
\end{tikzpicture}
};
\draw [->] (4.4,-0.3) -- (1.8,-0.3);
%\draw (2,2.7)  .. controls (2.6, 2.7) and (3, 2.3)  ..  (3.2,2.3);
%\draw (2,1.3)  .. controls (2.6, 1.3) and (3, 1.7)  ..  (3.2,1.7);
%\draw (2,2.7)  .. controls (0.8,2.7) and (0.8, 1.3)  ..  (2,1.3);
%\draw   (1.63,2) arc (135:45:0.7cm);
%\draw   (1.63,2) arc (-135:-35:0.7cm);
%%\draw     (1.63,2) arc (60:120:0.53cm);
%%\draw  [densely dotted](1.63,2) arc (-60:-120:0.53cm);
%\draw [densely dotted]   (3.2,2.3) .. controls (3.1, 2.1) and  (3.1, 1.9) ..  (3.2,1.7);
%\draw  (3.2,2.3) .. controls (3.3, 2.1) and  (3.3, 1.9) ..  (3.2,1.7);
%\draw (4.4,2.7)  .. controls (3.8, 2.7) and (3.4, 2.3)  ..  (3.2,2.3);
%\draw (4.4,1.3)  .. controls (3.8, 1.3) and (3.4, 1.7)  ..  (3.2,1.7);
%\draw (4.4,2.7)  .. controls (5.6,2.7) and (5.6, 1.3)  ..  (4.4,1.3);
%\draw   (4.77,2) arc (45:135:0.7cm);
%\draw   (4.77,2) arc (-45:-145:0.7cm);
%\draw     (4.77,2) arc (120:60:0.53cm);
%\draw  [densely dotted](4.77,2) arc (-120:-60:0.53cm);
%\node [vertex] at (2.6,6) (nodeB) {};
%\node [vertex] at (3.7,6) (nodeC) {};
\node [vertex] at (1.1,4.2) (nodeH) {};
\node [vertex] at (1.8,4.2) (nodeE) {};
\node at (1.45,4.03) () {$1$};
\node at (1.7,4.4) () {$v_2$};
\node at (2.3,3.8) () {$1$};
\node at (2.3,4.6) () {$1$};
\node at (4,3.7) () {$-1$};
\node at (4,4.7) () {$-1$};
\node [vertex] at (2.8,4.2) (nodeA) {};
\node at (3.1,4.03) () {$0$};
\node [vertex] at (3.5,4.2) (nodeD) {};
\node [vertex] at (4.5,4.2) (nodeF) {};
\node at (4.6,4.4) () {$v_1$};
\node [vertex] at (5.2,4.2) (nodeK) {};
\node at (4.85,3.9) () {$-1$};
%\draw (nodeB)  .. controls +(-1.3,-0.7) and +(-1.3,0.7) ..  (nodeB);
%\draw (nodeC)  .. controls +(1.3,-0.7) and +(1.3,0.7) ..  (nodeC);
\draw[->-] (nodeA)  .. controls +(-0.2,-0.3) and +(0.2,-0.3) ..  (nodeE);
\draw[->-] (nodeA)  .. controls +(-0.2,0.3) and +(0.2,0.3) ..  (nodeE);
\draw[-<-] (nodeD)  .. controls +(0.2,-0.3) and +(-0.2,-0.3) ..  (nodeF);
\draw[-<-] (nodeD)  .. controls +(0.2,0.3) and +(-0.2,0.3) ..  (nodeF);
%\draw (nodeB) -- (nodeC);
\draw [->-](nodeD) -- (nodeA);
\draw [->-](nodeK) -- (nodeF);
\draw [->-](nodeE) -- (nodeH);
%\node at (1.5,2.3)(nodeG) { \small $C_1$};
%\node at (3.2,2.6)(nodeI) { \small $C_2$};
%\node at (4.9,2.3)(nodeJ) { \small $C_3$};
%\fill (nodeB) circle [radius=1.5pt];
%\fill (nodeC) circle [radius=1.5pt];
\fill (nodeA) circle [radius=1.5pt];
\fill (nodeD) circle [radius=1.5pt];
\fill (nodeE) circle [radius=1.5pt];
\fill (nodeF) circle [radius=1.5pt];
\fill (nodeH) circle [radius=1.5pt];
\fill (nodeK) circle [radius=1.5pt];
%\node at (4.6,3) (nodeL) {$M$};
\node at (3.1,-0.5) (nodeL) {$F$};
\node at (1.45,3.3) (nodeL) {$\Gamma_F$};
%\node at (3.1,6.4) (nodeL) {$\bar \Gamma_F$};
\node at (3.2, 8) () {
\begin{tikzpicture}[thick, scale = 1.3]
\draw(10,3) -- (11,3) -- (11,2) -- (10,2) -- cycle;
 \fill[opacity = 0.3] (10,3) -- (11,3) -- (11,2) -- (10,2) -- cycle;
 \node at (9.9,3.15) () {$-1$};
%  \node at (9.9,3.15) () {$-1$};
%    \node at (11.2,2) () {$-1$};
        \node at (11.3,2) () {$-1$};
 \draw [->] (9.5,3) -- (12,3);
 \node at (11.7, 2.8) () {$c_1(v_1)$};
  \node at (11.5, 3.9) () {$c_1(v_2)$};
    \node at (10, 3.9) () {$\Hom_1(\Gamma_F, \R)$};
 \draw [->] (11,1.5) -- (11,4);
 %\node at (10, 2)  () {steady flows};
 %    \draw [decoration={text along path,
 %   text={steady flows},text align={center}, raise = -0.1cm},decorate]    (9.5,3) -- (11,1.5); 
 \end{tikzpicture}
 };
\end{tikzpicture}
}
\caption{An example of a vorticity $F$ on a pretzel, its measured Reeb graph, and the corresponding polytope of totally negative circulation functions.}\label{pretzel2}
\end{figure}

 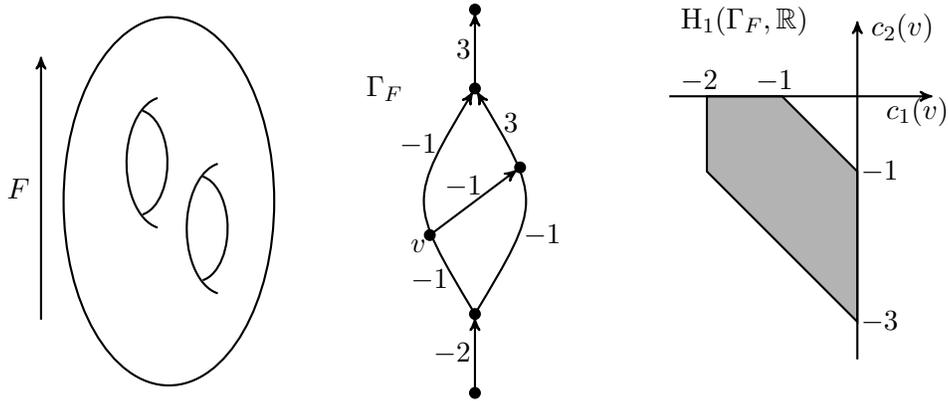
\begin{figure}[t]
\centerline{
\begin{tikzpicture}[thick, scale = 1.5]
  \node at (2.5,2) ()
  {
  \begin{tikzpicture}[thick, xscale = 1, yscale =1.75]
    \draw (2.4,1.9) ellipse (1.4cm and 1.4cm);
%    \draw   (2.85,1.13) arc (260:100:0.8cm);
%    \draw   (2.6,1.23) arc (-80:80:0.7cm);
%    \draw   (1.85,2.28) arc (260:100:0.5cm);
%    \draw   (1.65,2.38) arc (-80:80:0.4cm);
        \draw   (2.25,1.7) arc (260:100:0.5cm);
    \draw   (2.05,1.8) arc (-80:80:0.4cm);
        \draw   (3.05,1.2) arc (260:100:0.5cm);
    \draw   (2.85,1.3) arc (-80:80:0.4cm);
        \draw  [->] (0.7,1) -- (0.7,3);
    \node at (0.4,2) (nodeA) {$F$};
    \end{tikzpicture}
    };
%    \draw   (1.1,1.05) arc (-120:-60:0.95cm);
%    \draw   (2.05,1.05) arc (-120:-60:0.85cm);
%    \draw  [densely dashed] (2.05,1.05) arc (60:120:0.95cm);
%    \draw  [densely dashed] (2.9,1.05) arc (60:120:0.85cm);
    %
%    \draw   (1.1,2.75) arc (-120:-60:0.95cm);
%    \draw  (2.05,2.75) arc (-120:-60:0.85cm);
%    \draw  [densely dashed] (2.05,2.75) arc (60:120:0.95cm);
%    \draw  [densely dashed] (2.9,2.75) arc (60:120:0.85cm);
    \node [vertex] at (5.5,0.3) (nodeC) {};
    \node [vertex]  at (5.5,1) (nodeD) {};
    \node at (5.3,0.65) () {{$-2$}};
        \node at (5.1,1.3) () {$-1$};
                \node at (6.1,1.7) () {$-1$};
                  \node at (5.4,2.13) () {$-1$};
            %                    \node at (5.35,2) () {$0$};
               %                   \node at (4.85,2) () {$0$};
                                %  \draw [->, dashed]  (7,2) -- (8,2);
                                          \node at (5,2.5) () {$-1$};
                                              \node at (5.82,2.7) () {$3$};
                                              \node at (5.4,3.35) () {$3$};
    \node [vertex] at (5.5,3) (nodeE) {};
    \node [vertex]  at (5.5,3.7) (nodeF) {};
    \node [vertex]  at (5.1,1.7) (nodeG) {};    
        \node []  at (5,1.6) () {$v$};    
       % \node [vertex]  at (5.1,1.7) (nodeH) {};    
                \node [vertex]  at (5.9,2.3) (nodeI) {};    
    \draw  [a] (nodeC) -- (nodeD);
    \fill (nodeC) circle [radius=1.5pt];
    \fill (nodeD) circle [radius=1.5pt];
    \fill (nodeE) circle [radius=1.5pt];
    \fill (nodeF) circle [radius=1.5pt];
        \fill (nodeG) circle [radius=1.5pt];
       %     \fill (nodeH) circle [radius=1.5pt];
                \fill (nodeI) circle [radius=1.5pt];
    \draw  [a] (nodeE) -- (nodeF);
    \draw[a] (nodeG)  --  (nodeI);
     %   \draw[a] (nodeG)  .. controls +(0.2,+0.3) ..  (nodeH);
          %      \draw[a] (nodeD) -- (nodeG);
      %  \draw[a] (nodeG)  .. controls +(-0.1,+0.3) and +(-0.3,-0.3) .. (nodeE);
     \draw[a] (nodeD)  .. controls +(0.6,+1) ..  (nodeE);     
          \draw[a] (nodeD)  .. controls +(-0.6,+1) ..  (nodeE);     
             \node at (4.7, 3) () {$\Gamma_F$};  
\node at (8.5,2.2) () {
\begin{tikzpicture}[thick, scale = 1]
\draw (9,3) -- (10,3) -- (11,2) -- (11,0) -- (9,2) -- cycle;
 \fill[opacity = 0.3]  (9,3) -- (10,3) -- (11,2) -- (11,0) -- (9,2) -- cycle;
 \node at (8.9,3.2) () {$-2$};
  \node at (9.9,3.2) () {$-1$};
    \node at (11.3,2) () {$-1$};
        \node at (11.3,0) () {$-3$};
 \draw [->] (8.5,3) -- (12,3);
 \node at (11.8, 2.8) () {$c_1(v)$};
  \node at (11.6, 3.9) () {$c_2(v)$};
    \node at (9.5, 4) () {$\Hom_1(\Gamma_F, \R)$};
 \draw [->] (11,-0.5) -- (11,4);
%     \draw [decoration={text along path,
%    text={steady flows},text align={center}, raise = -0.1cm},decorate]    (9.5,3) -- (11,1.5); 
 \end{tikzpicture}
 };
\end{tikzpicture}
}
\caption{A vorticity $F$ on a pretzel with a different measured Reeb graph and the corresponding polytope of totally negative circulation functions.}\label{pretzel3}
\end{figure}

\begin{example}
Figures \ref{pretzel}, \ref{pretzel2}, \ref{pretzel3} show several vorticity functions on a pretzel with different measured Reeb graphs $\Gamma_F$, as well as the corresponding polytopes of totally negative circulation functions. The numbers at edges are integrals $\int_e f\diff \mu$. In Figures \ref{pretzel} and \ref{pretzel3}, the set of circulation functions on $\Gamma_F$ is parametrized by limits $c_1(v), c_2(v)$ at branches of the vertex $v$. In Figure \ref{pretzel2}, the set of circulation functions is parametrized by limits $c_1(v_1), c_1(v_2)$ at left branches of the vertices $v_1$, $v_2$.
\end{example}
%\begin{example}
%Figure \ref{pretzel2} shows another Morse function $F$ on a pretzel, its measured Reeb graph $\Gamma_F$, and the corresponding polytope of totally negative circulation functions. The set of circulation functions on $\Gamma_F$ is parametrized by limits $c_1(v_1), c_1(v_2)$ at left branches of the vertices $v_1$, $v_2$.
%\end{example}
\bigskip

%%%%%%%%%%%%%%%%%%%%%%%%%%

%%%%%%%%%%%%%%%%%%%%%%%%%%%%%%%%%%%%%

\section{Appendix A: Existence of steady triples in the vicinity of  hyperbolic levels}\label{app:proof-eight}.

In this appendix we prove technical Proposition \ref{step3} that near any figure-eight level of the vorticity function 
there exists a steady triple.

\tikzstyle{vertex} = [coordinate]
 \begin{figure}[t]
\centerline{
\begin{tikzpicture}[thick, scale = 1.75]
\node [vertex] (a) at (0,0) {};
\node [vertex] (z1) at (-0.5,0.3) {};
\node [vertex] (w11) at (-0.82,0.44) {};
\node [vertex] (w12) at (-0.73,-0.4) {};
\node [vertex] (z2) at (-0.5,-0.3) {};
\node [vertex] (z3) at (0.5,-0.3) {};
\node [vertex] (z4) at (0.5,0.3) {};
\fill (a) circle [radius=1.5pt];
\fill (z1) circle [radius=1.5pt];
\fill (w11) circle [radius=1.5pt];
\fill (w12) circle [radius=1.5pt];
\fill (z2) circle [radius=1.5pt];
\fill (z3) circle [radius=1.5pt];
\fill (z4) circle [radius=1.5pt];
\begin{pgfinterruptboundingbox}
\draw [->--, -->-] (a) .. controls +(-3,2) and +(-3,-2) .. (a);
\draw  [->--, -->-]  (a) .. controls +(3,-2) and +(3,2) .. (a);
\draw (0,0.3)  .. controls +(-3.4,2) and +(-3.4,-2) .. (0,-0.3);
\draw (0,0.3)  .. controls +(3.4,2) and +(3.4,-2) .. (0,-0.3);
\draw (-0.4, 0)  .. controls +(-2,1.3) and +(-2,-1.3) .. (-0.4, 0);
\draw (0.4, 0)  .. controls +(2,1.3) and +(2,-1.3) .. (0.4, 0);
  \fill[opacity = 0.2] (0,0.3)  .. controls +(-3.4,2) and +(-3.4,-2) .. (0,-0.3);
    \fill[opacity = 0.2] (0,0.3)  .. controls +(3.4,2) and +(3.4,-2) .. (0,-0.3);
      \fill[color = white, opacity = 1] (-0.4, 0)  .. controls +(-2,1.3) and +(-2,-1.3) .. (-0.4, 0);
            \fill[color = white, opacity = 1] (0.4, 0)  .. controls +(2,1.3) and +(2,-1.3) .. (0.4, 0);
            \end{pgfinterruptboundingbox}
            \node () at (-2.1,0) {$\ell_1$};
                        \node () at (-2.3,0.3) {$U_1$};
                \node () at (2.1,0) {$\ell_2$};
                   \node () at (2.35,0.3) {$U_2$};
                \node () at (0.02,0.18) {$O$};    
                \fill[opacity = 0.2]  (a) circle [radius=35pt];      
                \node () at (0,1) {$U$};    
                \node () at (-0.3,0.33) {$Z_{1}$};
                      \node () at (-1,0.4) {$W_{1}$};
                            \node () at (-0.95,-0.3) {$W_{3}$};
                          \node () at (-0.32,-0.33) {$Z_{3}$};
                            \node () at (0.7,-0.28) {$Z_{4}$};
                              \node () at (0.7,0.28) {$Z_{2}$};
                        \node (R11) at (-0.5,1.4) {$R_{1}$};
                       \draw [->] (R11) -- (-0.5, 0.4);
                                               \node (R12) at (-0.5,-1.4) {$R_{3}$};
                       \draw [->] (R12) -- (-0.5, -0.4);
                                       \node (R21) at (0.4,1.4) {$R_{2}$};
                       \draw [->] (R21) -- (0.4, 0.35);
                                                              \node (R22) at (0.4,-1.4) {$R_{4}$};
                       \draw [->] (R22) -- (0.4, -0.35);
                                                                                     \node (U0) at (0,-1.5) {$U_0$};
                       \draw [->] (U0) -- (0, -0.15);
                \draw (-0.8, 0.23) .. controls +(-0.05,+0.22) .. (-0.7, 0.64);
                \draw (-0.2, 0.41) .. controls +(0.1, -0.2) .. (-0.5,0.07);
                       \draw (-0.8, -0.23) .. controls +(0.1,-0.22) .. (-0.7, -0.64);
                        \draw  (-0.15, -0.38) .. controls +(0.1, 0.2) .. (-0.5,-0.07);
                           \draw (0.85, 0.25) .. controls +(0.05,+0.22) .. (0.7, 0.64);
                           \draw (0.2, 0.41) .. controls +(0, -0.2) .. (0.5,0.07);
                              \draw (0.85, -0.25) .. controls +(0.1,-0.22) .. (0.7, -0.64);
                               \draw (0.2, -0.41) .. controls +(0, 0.2) .. (0.5,-0.07);
\end{tikzpicture}
}
\caption{Constructing a steady flow near a hyperbolic fiber. }\label{hyperbolicNeighborhood}
\end{figure}
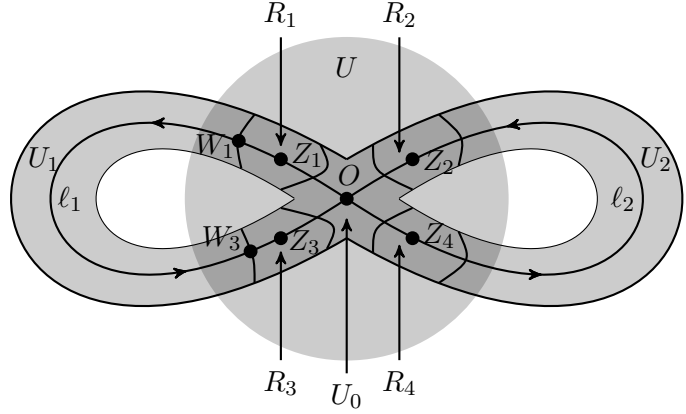

\begin{proposition}{\bf (=\ref{step3})}
Let $(M, \omega)$ be a compact symplectic surface, possibly with boundary, and let $F \colon M \to \R$  be a simple Morse function on $M$. Let also $\pazocal E$ be a figure-eight level of $F$. Then there exists a steady $F$-triple $(\alpha, J, H)$ in the neighborhood of $\pazocal E$. Moreover, for any non-zero numbers $c_1, c_2 \in \R$ having the same sign, the triple $(\alpha, J, H)$ can be chosen in such a way that
$$
\int_{\ell_i} \alpha = c_i\,,
$$
where $\ell_i$ are the loops constituting the level $\pazocal E$ (see Figure \ref{hyperbolicNeighborhood}).
\end{proposition}

\begin{proof}
First, we apply Lemma \ref{step2} to construct an $F$-steady triple $(\alpha, J, H)$ 
such that $ \sgn (\diff H/\diff F) = \eps := -\sgn c_1 = -\sgn c_2\, $
in the neighborhood $U$ of the singular point $O \in \pazocal E$. Since $H$ is a function of $F$, it extends to the neighborhood of $\pazocal E$ in a unique way.\par
Consider points $Z_{1}, Z_{2}, Z_{3}, Z_{4} \in U \cap \pazocal E$ depicted in Figure \ref{hyperbolicNeighborhood}. Note that in $U$ we have 
\begin{align}\label{ori}
\omega^{-1}(\alpha, \diff H) = \omega^{-1}(J^*\alpha, \alpha) = g^{-1}(\alpha, \alpha)\,,
\end{align}
where $g$ is the Riemannian metric given by $g(u_1,u_2) = \omega(u_1, Ju_2)$. Therefore, we have
$
\omega^{-1}(\alpha, \diff H) > 0
$
at all points except the point $O$. In particular, the forms $\alpha$ and $\diff H$ are linearly independent at the 
points $Z_{i}$. The latter implies that there exists  functions $G_{i}$ in  neighborhoods of points $Z_{i}$ which form a 
positive coordinate system with $F$ (i.e.  $   (\diff G_{i} \wedge \diff F)/\omega> 0$) and such that
%\begin{align}\label{alphainR}
$\alpha = A_{i}(F, G_{i}) \,\diff G_{i}$
%\end{align} 
for certain smooth functions $A_{i}$. 
\par
Let $R_{i}$ be a small rectangle in $(F,G_{i})$ coordinates near $Z_{i}$. Taking, if necessary, a smaller neighborhood 
of $\pazocal E$, we may assume that the complement to $R_{i}, \, i=1,...,4$ in this neighborhood consists 
of three regions $U_0, U_1,U_2$ where $O \in U_0$  (see Figure \ref{hyperbolicNeighborhood}). 
\par
 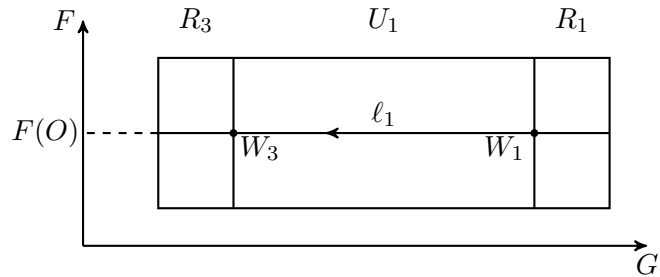
\begin{figure}[b]
\centerline{
\begin{tikzpicture}[thick, scale = 1]
\draw (0,0) -- (6,0) -- (6,2) -- (0,2) -- cycle;
\draw (1,0) -- (1,2);
\draw (5,0) -- (5,2);
\node () at (0.5, 2.5) {$R_{3}$};
\node () at (3, 2.5) {$U_{1}$};
\node () at (3, 1.25) {$\ell_{1}$};
\node () at (5.5, 2.5) {$R_{1}$};
\draw [->] (-1,-0.5) -- (-1, 2.5);
\node () at (-1.25,2.5) {$F$};
\draw [->] (-1,-0.5) -- (6.5, -0.5);
\node () at (6.5,-0.75) {$G$};
\draw [-<-] (0,1) -- (6,1);
\draw [dashed] (0,1) -- (-1, 1);
\fill (1,1) circle [radius=1.5pt];
\fill (5,1) circle [radius=1.5pt];
\node () at (-1.5,1) {$F(O)$};
\node () at (1.35,0.8) {$W_{3}$};
\node () at (4.6,0.8) {$W_{1}$};
\end{tikzpicture}
}
\caption{Region $V =  U_1 \cup R_{1} \cup R_{3}$. }\label{regionV1}
\end{figure}
Consider the region $V = U_1 \cup R_{1} \cup R_{3}$. Extend the functions $G_{1}$ and $G_{3}$ (maybe shifting 
one of them by a constant)  from $R_{1}$ and  $R_{3}$  to a function $G$ defined in the region $V$  in such a way 
that $(G, F)$ is a positive coordinate system in $U_1$. (This means that the orientation defined by $\diff G \wedge \diff F$ 
coincides with the symplectic orientation or the function $G$ decreases along $\ell_1$.)
In coordinates $ (G, F)$, the region $V$ is a rectangle depicted in Figure \ref{regionV1}. Our aim is to extend the 1-form 
$\alpha$ defined in $R_{1}$, $R_{3}$ to $V$.  First, note that in regions $R_{1}, R_{3}$ we have
\begin{align}\label{ori2}
\omega^{-1}(\alpha, \diff F) =A_{i}(F, G) \,\omega^{-1}(\diff G, \diff F)
=-A_{i}(F, G) \,\{F,G\}\,,
\end{align}
where  $\{F,G\}$ is the Poisson bracket associated with the symplectic structure~$\omega$. This, in particular, implies that 
$\sgn A_{i} = \sgn (\diff H/\diff F) = \eps\,.$
%Since the orientation defined by $\diff G_{1} \wedge \diff F$ coincides with the symplectic orientation, we have $ \omega^{-1}(\diff G_{1}, \diff F) > 0\,,$ so from formulas \eqref{ori} and \eqref{ori2} it follows that $$\sgn A_{ij} = \sgn \diffFXp{H}{F} = \eps\,.$$ 
We first extend the functions $A_{1}(F(O), G)$ and  $A_{3}(F(O), G)$ (where $F(O)$ is the value of $F$ 
at the level $\pazocal E$)  defined on arcs $\ell_1 \cap R_{1}$ and $\ell_1 \cap R_{3}$ 
 to a function $A^0(G)$ defined on $\ell_1 \cap V$ and such that $\sgn A^0 = \sgn A_{i} = \eps$. Now define 
 a function
 $$
A(F, G) :=A^0(G) +\! \!\!\int\limits_{F(O)}^F\!\! \frac{F}{\{F,G\}}\diff F\,.
$$
We claim that this function coincides with $A_i$ in the regions $R_i$. Indeed, by taking the exterior differential 
in the equation $\alpha = A_{i}(F, G_{i}) \,\diff G_{i}$,
% \eqref{alphainR}, 
we get $F\omega = \diffX{F}{A_{1}(F, G)} \,\diff F \wedge \diff G$, i.e., 
$F = \diffX{F}{A_{1}(F, G)}{} \cdot \{F,G\}\,.$
Then the corresponding derivatives coincide:
\begin{align}\label{eqDer}
\diffX{F}{A_{1}(F, G)}=  \frac{F}{\{F,G\}} = \diffX{F}{A(F, G)}\,,
\end{align}
and together with $A(F(O), G) =A^0(G)= A_{1}(F(O), G),$ equation \eqref{eqDer} implies that $A = A_{1}$ 
in the region $R_{1}$. Similarly, $A = A_{3}$ in $R_{3}$. Now, we extend the 1-form $\alpha$ to $V$ by setting
$\alpha := A \,\diff G.$
Then $\diff \alpha = F\omega$, as desired. Repeating the same procedure for the region 
$\tilde V=U_2 \cup R_{2} \cup R_{4}$, 
we obtain a 1-form $\alpha$ defined in the whole neighborhood of $\pazocal E$ and satisfying 
$\diff \alpha = F\omega$. Furthermore,  in $V$ we have
$$
\omega^{-1}(\alpha, \diff H) = \diffFXp{H}{F}\, A(F, G) \,\omega^{-1}(\diff G, \diff F) > 0
$$
 in a sufficiently small neighborhood of $\pazocal E \cap V$ and, similarly, for $\pazocal E \cap \tilde V$. 
 This allows one to extend the almost complex structure $J$ to a neighborhood of $\pazocal E$ by setting
%\begin{align*}
$J^*\alpha := -\diff  H,\,  J^*\diff H := \alpha\,.$
%\end{align*}
Thus, the first statement is proved. 

To prove the second statement denote by $W_{i}$ be the intersection point of $\ell_1$ 
and the common boundary of the regions $R_{i}$ and $U_1$ 
(see Figures~\ref{hyperbolicNeighborhood},~\ref{regionV1}). Then
\begin{align}\label{partialIntegral}
\int\limits_{W_{1}}^{W_{3}} \alpha \,\, = \!\!\!\int\limits_{G(W_{1})}^{G(W_{3})}\!\!\!A^0(G)\,\diff G\,.
\end{align}
By using ambiguity in the choice of a smooth function $A^0(G)$  extending the functions $A_{i}(F(O), G)$ and satisfying 
$\sgn A^0 = \eps$, one can choose it so that the integral \eqref{partialIntegral} is equal to any $c$ (note that 
$G(W_{1})>G(W_{3})$) provided that  $\sgn c = -\eps$. Furthermore,  by moving the points $Z_{1}, Z_{3}$ 
closer to $O$ and shrinking the regions $R_{1}, R_{3}$, one can make the integral~\eqref{partialIntegral} arbitrary close to 
$\int_{\ell_1} \alpha$. Since $\sgn c_1 = -\eps$, this implies that one can choose $\alpha$ such that
$\int_{\ell_1} \alpha = c_1\,$. The proof for $\ell_2$ is analogous.
\end{proof}

\section{Appendix B: Casimir invariants of the 2D Euler equation}
Above we classified coadjoint orbits of the group $\SDiff(M)$ in terms of graphs with some additional structures, 
see Theorem \ref{thm:sdiffM}. However, for applications, it is important to describe numerical invariants of the coadjoint action, i.e., Casimir functions. We begin with the description of such invariants for functions on symplectic surfaces.\par
Let $(M, \omega)$ be a compact connected symplectic surface, possibly with boundary, and let $F$ be a simple Morse function on $M$. With each edge $e$ of the measured Reeb graph $\Gamma_F = (\Gamma, f ,\mu)$, one can associate an infinite sequence of moments
$$
m_{i,e}(F) = \int_e f^i \,\diff \mu = \int_{M_e}\!\! F^i \,\omega\,,
$$
where $ i =0,1,2,\dots$, and $M_e = \pi^{-1}(e)$ for the natural projection $\pi \colon M \to \Gamma$. Obviously, the moments $m_{i,e}(F) $ are invariant under the action of $\SDiff(M)$ on simple Morse functions. Moreover, they form a complete set of invariants in the following sense:
\begin{theorem}
Let $(M, \omega)$ be a compact connected symplectic surface, possibly with boundary, and let $F$ and $G$ be simple Morse functions on $M$. Assume that $\phi \colon \Gamma_F \to \Gamma_G$ is an isomorphism of abstract directed graphs which preserves moments on all edges. Then $\Gamma_F$ and $\Gamma_G$ are isomorphic as measured Reeb graphs, and there exists a symplectomorphism $\Phi \colon M \to M$ such that $\Phi_*F = G$.
\end{theorem}

\begin{proof}
Consider an edge $e = [v ,w] \in \Gamma_F$. Pushing forward the measure $\mu$ on $e$ by means of the homeomorphism $f \colon e \to [f(v), f(w)] \subset \R$, we obtain a measure $\mu_f$ on the interval $I_f = [f(v), f(w)]$, whose moments coincide with the moments of $\mu$ at $e$. Repeating the same construction for the measure on $\Gamma_G$, we obtain another measure $\mu_g$, which is defined on the interval $I_g = [g(\phi(v)), g(\phi(w))]$ and has the same moments as $\mu_f$. \par Now, consider any closed interval $I \subset \R$ which contains both $I_f$ and $I_g$. Then the measures $\mu_f$, $\mu_g$ may be viewed as measures on the interval $I$ supported at $I_f$ and $I_g$ respectively. The moments of the measures $\mu_f, \mu_g$ on $I$ coincide, so, by the uniqueness theorem for the Hausdorff moment problem, we have $\mu_f = \mu_g$, which implies the proposition.
\end{proof}

The above theorem allows one to describe Casimirs of the 2D Euler equation on $M$, i.e. invariants of the coadjoint action 
of the symplectomorphism group $\SDiff(M)$. 
Let $F$ be a Morse vorticity function of an ideal flow with velocity $v$ on a surface $M$, possibly with boundary, and let
$\Gamma$ be its Reeb graph. The corresponding moments  $m_{i,e}(F) $ for this vorticity are natural to call 
{\it generalized enstrophies}. Then group coadjoint orbits in the vicinity of an
orbit with the vorticity function $F$ are singled out as follows.

\begin{corollary}\label{cor:Casimirs}
A  complete set of Casimirs of a 2D Euler equation in a neighborhood of a Morse-type coadjoint orbit is  given 
by the moments $m_{i,e}$ for each edge $e \in \Gamma$,  $i=0,1,2,\dots$, and all circulations of the velocity $v$
over cycles in the singular levels of $F$ on $M$.
\end{corollary}

Note that the (finite) set of required circulations can be sharpened by considering fewer quantities 
needed to describe the circulation function, as in Section \ref{sect:coadj-sdiff}.

%\bk{Old version:
%Now, let us describe invariants of the coadjoint action of $\SDiff(M)$. Consider the set $U_\Gamma \subset \SVect^*(M)$ of simple cosets $[\alpha]$ such that the Reeb graph of $\Diff [\alpha]$ coincides with a given graph $\Gamma$ as an abstract directed graph. The set $U_\Gamma$ is an open subset of $\SVect^*(M)$ invariant with respect to the $\SDiff(M)$ action.
%\begin{corollary}
%Let $(M, \omega)$ be a compact connected symplectic surface, possibly with boundary. Then a complete set of invariants for the coadjoint action of $\SDiff(M)$ on $U_\Gamma$ is given by the moments $m_i(e)$ for each edge $e \in \Gamma$, and a $1$-chain $\gamma$ on $\Gamma$ satisfying $\partial \gamma = \chi \mod \partial \Gamma$, where the $0$-chain $\chi$ is given by 
%$$ \chi := -\sum_{{v \, \in \, V}}\sum_{{e \leftarrow v}} m_1(e)v\,.$$
%\ai{In principle, one can explicitly parametrize the set of such chains, however there is no canonical way to do it, so may be it is better not to do it at all.}.
%\end{corollary} }

\begin{remark}
As invariants of the coadjoint action of $\SDiff(M)$, one usually considers total moments
$$
m_i(F)= \int_{M}\!\! F^i \,\omega\ = \int_{\Gamma} f^i \,\diff \mu \, ,
$$
where $F = \Diff[\alpha]$ is the vorticity function, and $(\Gamma,f,\mu)$ is the measured Reeb graph of $F$. However, the latter moments do not form a complete set of invariants even in the case of a sphere or a disk.

 \begin{figure}[t]
 \centering
\begin{tikzpicture}[thick, scale = 1.5]
\node [vertex] (A) at (4.9,0.7) {};
\node[vertex] (B) at (5.5,-0.3) {};
\node[] () at (4.5,0.6) {$\Gamma$};
\node[vertex] (C) at (6,1) {};
\node[vertex] (D) at (5.5,-1) {};
    \fill (A) circle [radius=1.5pt];
    \fill (B) circle [radius=1.5pt];
        \fill (C) circle [radius=1.5pt];
            \fill (D) circle [radius=1.5pt];
    \draw [a] (B) -- (A);
        \draw [a] (B) -- (C);
            \draw [a] (D) -- (B);
            \draw [->] (8,-1.2) -- (8,1.5);
            \draw [dotted] (5.32, 0) -- (8,0);
                  \draw [dotted] (5.1, 0.4) -- (8,0.4);
                 % \draw [dashed, ->] (6.2,-0.2) -- (7.2,-0.2);
                  %   \fill (5.3, 0) circle [radius=1pt];
               %   \node () at (5.2, 0) {$a_1$};
                 %   \node () at (5, 0.4) {$b_1$};
                   \fill (8,0) circle [radius=1pt];
                   \draw (5.27,0) -- (5.37, 0);
                      \draw (5.57,0) -- (5.67, 0);
                      \draw (5.03,0.4) -- (5.13, 0.4);
                          \draw (5.73,0.4) -- (5.83, 0.4);
                    %   \fill (5.32,0) circle [radius=1pt];
                               %  \fill (5.6,0) circle [radius=1pt];
                     \fill (8,0.4) circle [radius=1pt];
                 \node () at (8.15, 1.3) {$f$};
                 \node () at (8.12, 0) {$a$};
                       \node () at (8.12, 0.4) {$b$};
                         \node () at (5, 0.2) {$I_1$};
                          \node () at (5.9, 0.2) {$I_2$};
\end{tikzpicture}
\caption{Modifying the measure on the edges.}\label{modMeasure}
\end{figure}
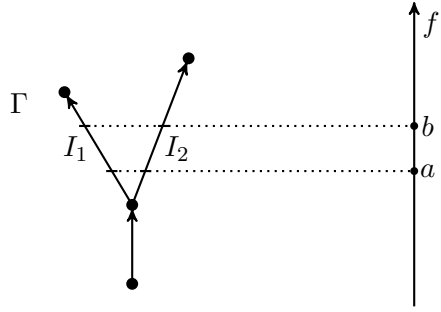
Consider, for example, the measured Reeb graph $(\Gamma, f ,\mu)$ depicted in Figure \ref{modMeasure}. Let $\mu'$ be any smooth measure on $\R$ supported in $[a,b]$. Define a new measure $\tilde \mu$ on $\Gamma$ by ``moving some density from one branch to another", i.e. by setting
\begin{align*}
\tilde \mu := \begin{cases} \mu + f^*(\mu') \mbox{ in } I_1,\\
 \mu - f^*(\mu') \mbox{ in } I_2,
\end{cases}
\end{align*}
and $\tilde \mu := \mu$ elsewhere. Then $(\Gamma, f ,\tilde \mu)$ is a again a measured Reeb graph. Moreover, for all total moments we have
$$
\int_\Gamma f^k\,\diff \mu = \int_\Gamma f^k\,\diff \tilde \mu\,.
$$
However, the graphs $(\Gamma, f ,\mu)$ and $(\Gamma, f ,\tilde \mu)$ are not isomorphic and thus correspond to two different coadjoint orbits of $\SDiff(S^2)$.
\end{remark}

\medskip 
\bibliographystyle{plain}
\bibliography{sympl_orbits}

\end{document}